\documentclass[12pt,leqno]{amsart}
\usepackage{amsmath,color, multirow,appendix}
\usepackage[dvipsnames]{xcolor}
\usepackage[hyperindex=true,frenchlinks=true,colorlinks=true,
citecolor=Mahogany,linkcolor=Red,urlcolor=Tan,linktocpage]{hyperref}
\usepackage{graphicx}
\usepackage[authoryear]{natbib}

\numberwithin{equation}{section}

\newcommand{\R}{{\mathbb R}}
\newcommand{\Z}{{\mathbb Z}}
\newcommand{\N}{{\mathbb N}}

\newcommand{\F}{{\mathcal F}}
\newcommand{\Var}{{\rm Var}}

\newcommand{\pr}[1]{\left(#1\right)}
\newcommand \abs[1]{\left|#1\right|}
\newcommand{\E}[1]{\mathbb E\left[#1\right]}
\newcommand{\var}[1]{\operatorname{Var}\left(#1\right)}
\newcommand{\cov}[2]{\operatorname{Cov}\left(#1,#2\right)}
\newcommand{\norm}[1]{\left\lVert #1 \right\rVert}
\newcommand \ens[1]{\left\{ #1\right\}}
\newcommand{\1}[1]{\mathbf{1}_{#1}}
\newcommand \PP{\mathbb P}
\newcommand{\Fca}{\mathcal{F}}

\newtheorem{lemma}{Lemma}[section]

\newtheorem{theorem}{Theorem}
\newtheorem{proposition}[lemma]{Proposition}

\newtheorem{corollary}[lemma]{Corollary}
\newtheorem{example}[lemma]{Example}

\theoremstyle{definition}
\newtheorem{remark}[lemma]{Remark}
\begin{document}

\title[U-statistics of local sample moments under weak dependence]
{U-statistics of local sample moments under weak dependence}
\author[H. Dehling]{Herold Dehling}
\author[D. Giraudo]{Davide Giraudo}
\author[S.~K.~Schmidt]{Sara K.~Schmidt}
\today

\address{
Fakult\"at f\"ur Mathematik, Ruhr-Universit\"at Bo\-chum, 
44780 Bochum, Germany}
\email{herold.dehling@rub.de}

\address{Institut de Recherche Math\'ematique Avanc\'ee, Universit\'e de Strasbourg, 7 rue Ren\'e Descartes, 67084 Strasbourg, France}
\email{dgiraudo@unistra.fr}

\address{
Fakult\"at f\"ur Mathematik, Ruhr-Universit\"at Bo\-chum, 
44780 Bochum, Germany}
\email{sara.schmidt@rub.de}


\thanks{Supported in part by the 
Collaborative Research centre 823, {\em Statistical Modelling of Nonlinear Dynamic Processes}, of the German Research 
Foundation. S. K. Schmidt was additionally supported by the Friedrich-Ebert-Stiftung. }
\begin{abstract}
In this paper, we study the asymptotic distribution of some U-statistics whose entries are functions of empirical moments computed from non-overlapping consecutive blocks of an underlying weakly dependent process. The length of these blocks converges to infinity, and thus we consider U-statistics of triangular arrays. We establish asymptotic normality of such U-statistics. The results can be used 
to construct tests for changes of higher order moments.
\end{abstract}
\maketitle
\allowdisplaybreaks
\bibliographystyle{plainnat}
\section{Introduction}
 
Given some real-valued data $Y_1,\ldots,Y_n$ and a symmetric measurable function $h:\R^m\rightarrow \R$, we define the $U$-statistic with kernel $h$ as
\[
 U_n:= U_n(h)=\frac{1}{\binom{n}{m}} \sum_{1\leq i_1<i_2 <\ldots <i_m \leq n} h(Y_{i_1},\ldots, Y_{i_m}).
\]
U-statistics play an important role in nonparametric statistics, as many sample statistics can be expressed in this way, at least asymptotically. Well-known examples include the sample variance, Gini's mean difference, the Cram\'er-von~Mises test statistic and the $\chi^2$-test statistic for goodness of fit. For details and further examples see e.g. \citet{MR0595165} and \citet{MR2283250}. U-statistics have been introduced independently by \citet{MR0015746} and \citet{MR0026294}. \citet{MR0015746} showed that for i.i.d. data, $U_n(h)$ is an unbiased estimator of the parameter $\theta= \E{h\pr{Y_1,\ldots,Y_m}}$, and that it is minimum variance unbiased in nonparametric models. \citet{MR0026294} proved that, again for i.i.d.\ data and for square integrable kernels, the U-statistic is asymptotically normal.
More precisely,
\[
 \sqrt{n}\pr{U_n(h)-\E{h\pr{Y_1,\ldots,Y_m} }} \longrightarrow N(0,4\, \gamma_h^2)
\]
in distribution, where $\gamma_h^2:=\Var\pr{\E{h(Y_1,\ldots,Y_m)|Y_2,\ldots,Y_m} }$. In the so-called degenerate case $\gamma_h^2=0$, a different normalization is required to get a non-trivial limit, which will be a non-normal distribution; see \citet{MR0595165} for further details. 
Most of the results for i.i.d.\ data can be extended to weakly dependent stationary processes $\pr{Y_i}_{i\geq 1}$; see e.g. \cite{MR2283250} for a survey.

In this paper, we study the asymptotic distribution of certain U-statistics whose entries are local summary statistics  of an underlying weakly dependent process $\pr{X_t}_{t\in \Z}$. More precisely, we consider local statistics $g\Big(\frac{1}{\ell_n} \sum_{t\in B_{n,j}} X_t, \frac{1}{\ell_n} \sum_{t\in B_{n,j}} X_t^2, \ldots, \frac{1}{\ell_n} \sum_{t\in B_{n,j}} X_t^m\Big),$ $1\leq j \leq b_n$, which can be expressed as a function of the first $m$ empirical moments of the consecutive non-overlapping blocks 
\begin{equation}\label{eq:def_Bnj}
B_{n,j}:=\{ (j-1) \ell_n+1,\ldots ,j\ell_n \}
\end{equation}
  for $1\leq j\leq b_n$. We assume the block length $\ell_n$ to converge to infinity. Given an appropriate scaling factor $\sqrt{\ell_n}$ and certain regularity assumptions on $g\colon \R^m\to\R$, we will show that the statistics 
\begin{equation}\label{eq: Specific triang array}
 W_{n,j}:=\sqrt{\ell_n} g\Big(\frac{1}{\ell_n} \sum_{t\in B_{n,j}} X_t, \frac{1}{\ell_n} \sum_{t\in B_{n,j}} X_t^2, \ldots, \frac{1}{\ell_n} \sum_{t\in B_{n,j}} X_t^m\Big)
\end{equation}
 are each asymptotically normal. We are then interested in U-statistics of the type
\begin{equation}
\label{eq:U-stat}
U_n:=\frac 1{b_n\pr{b_n-1}}\sum_{1\leq j\neq k\leq 
b_n}h\pr{W_{n,j},W_{n,k}}.
\end{equation}

Such U-statistics arise naturally in nonparametric tests for the constancy of parameters of the underlying process $(X_t)_{t\in\Z}$. \citet*{MR4352537} test for the constancy of the variance by analysing Gini's mean 
difference of the logarithmic local sample 
variances, i.e. they choose $h(x,y)=\abs{x-y}$ and 
$W_{n,j}=\sqrt{\ell_n}\log(\frac{1}{\ell_n}
\sum_{t=(j-1)\ell_n+1}^{j\ell_n} (X_t- \frac{1}
{\ell_n} \sum_{r=(j-1)\ell_n+1}^{jl\ell_n} X_r)^2)$. \citet{Schmidt_2022} tests for changes in the mean by considering Gini's mean difference of the local sample means $W_{n,j}=\frac{1}{\sqrt{\ell_n}}\sum_{t\in B_{n,j}} X_t$.  In both works, the behaviour of the test statistic under the hypothesis is determined by deriving a central limit theorem for $U_n$. The setup considered in the present paper allows for testing for constancy of higher order characteristics of the distribution of $X_t$, such as the skewness or kurtosis; see Example~\ref{Example:applications} below.

Note that the entries of $U_n$ from \eqref{eq:U-stat} stem from 
a triangular array $\pr{W_{n,j}}_{1\leq j\leq b_n, n\geq 1}$  
and that under certain regularity assumptions made in this paper, they each converge to a normal law as $\ell_n\rightarrow \infty$. 
Moreover, assuming the number $b_n$ of blocks to converge to infinity as well, there additionally holds (under appropriate assumptions) a central limit theorem for the U-statistic itself. The limit distribution of $U_n$ is hence determined by the double asymptotics of the U-statistic and its entries.
It is the goal of this paper to investigate more systematically such structures and to find minimal conditions that guarantee asymptotic normality of the resulting U-statistics of type \eqref{eq:U-stat}.

\section{Main results} 
We are interested in U-statistics of triangular arrays of the form  
\eqref{eq:U-stat}, where $h\colon \R^2\to\R$ denotes a symmetric kernel function. We will henceforth always assume that
the kernel fulfils
\begin{equation} \label{eq:condition_on_h}
 \abs{h\pr{x,y}}\leq C\pr{1+\abs{x}+\abs{y}}
\end{equation}
for all $ x,y\in\R$ and some constant $C$. For the results under dependence later on, we will require the stronger assumption of Lipschitz-continuity. 

Our first result is a central limit theorem for $U_n$, given the triangular array $(W_{n,j})_{1\leq j\leq b_n, n\geq 1}$ is row-wise i.i.d.\ with a very mild assumption  on the distribution of the random variables $W_{n,j}$.  

\begin{theorem}\label{thm:TLC_triangular_array}
 Let $\pr{W_{n,j}}_{1\leq j\leq b_n,n\geq 1}$ be a row-wise 
 i.i.d.\ triangular 
array such that $b_n\to\infty$ as $n\to\infty$. 
Assume that  \eqref{eq:condition_on_h} holds, that $
 \gamma_n^2:=\cov{h\pr{W_{n,1},W_{n,2}} }{h\pr{W_{n,2},W_{n,3}}}>0$, 
 and that the sequence 
$\pr{W_{n,1}^2/\gamma_n^2}_{n\geq 1}$ is uniformly integrable. Then the 
following convergence in distribution holds
\begin{equation}\label{eq:TLC_triangular_array}
 \frac{\sqrt{b_n}}{\gamma_n}\pr{ 
 \frac{1}{b_n\pr{b_n-1}}\sum_{1\leq j \neq k\leq b_n}h\pr{W_{n,j},W_{n,k}}
 -\E{h\pr{W_{n,1},W_{n,2}}}
 }\to N\pr{0,4}.
\end{equation}

\end{theorem}

A related result was obtained in \citet{MR4289844} for incomplete 
U-statistics of independent data.

The above theorem lays the groundwork for the more specific problems we investigate in this paper. 
As opposed to Theorem 1, the triangular array $(W_{n,j})_{1\leq j\leq b_n, n\geq 1} $  we consider from now on is in general not row-wise independent as the underlying process $\pr{X_t}_{t\in\Z}$ is weakly dependent. More specifically, we assume the stationary sequence $(X_t)_{t\in\Z}$ to be expressible as a functional of an i.i.d.~process. Thus, we can write $X_t:= 
f\pr{\pr{\varepsilon_{t-u}}_{u\in\Z}  }$, where $f\colon \R^\Z\to \R$ is measurable 
and $\pr{\varepsilon_u}_{u\in\Z}$ is i.i.d. In order to quantify the dependence, let 
$\pr{\varepsilon'_u}_{u\in\Z}$ be an independent copy of  $\pr{\varepsilon_u}_{u\in\Z}$ and 
define 
\begin{equation}\label{eq:definition_de_delta}
 \delta_{i}\pr{\pr{X_t}_{t\in\Z}}:=\norm{X_0-X_0^{*,i}}_2,
\end{equation}
where $X_0^{*,i}=f\pr{ \pr{\varepsilon^{*,i}_{-u  }}_{u\in\Z}}$ and $\varepsilon^{*,i}_v=\varepsilon'_i$ if 
$v=i$ 
and $\varepsilon^{*,i}_v=\varepsilon_v$ otherwise. We thus measure the contribution of 
$\varepsilon_i$ to $X_0$ by looking at the difference between $X_0$ and a coupled 
version $X_0^{*,i}$ for which $\varepsilon_i$ is replaced by an independent copy. This weak dependence concept was introduced by \citet{MR2172215} under the term \emph{physical dependence measure}  and is now frequently used in statistical applications (see, e.g., \citet{MR3225977}, \citet*{MR3114713}, \citet{MR2384990} and \citet{MR2812816}).

In the following, the triangular array  $(W_{n,j})_{1\leq j\leq b_n, n\geq 1} $  is assumed to be of the form \eqref{eq: Specific triang array}.
Example \ref{Example:applications} presents some problems that are covered by this structure. 
\begin{example}\label{Example:applications}
\begin{enumerate}
\item  \citet{MR4352537} propose a test for constancy of the variance based on the test statistic
\[
  U_n= \frac{1}{b_n(b_n-1)} \sum_{1\leq j,k\leq b_n} \sqrt{\ell_n}\big|\log s_{n,j}^2-\log s_{n,k}^2\big|,
\]
where $s_{n,j}^2:=\sum_{t\in B_{n,j}} (X_t- \frac{1}{\ell_n} \sum_{r\in B_{n,j}} X_r)^2 $.
In our setting, this corresponds to $m=2$, $g(x_1,x_2)=\log(x_2-x_1^2)$ and $h\pr{x,y}=\abs{x-y}$. 

\item  Considering higher moments, one can construct a test for the constancy of the skewness or the kurtosis in a similar fashion to (1) by considering Gini's mean difference (that is, $h\pr{x,y}=\abs{x-y}$) of the blockwise estimates $\hat{\gamma}_{n,j}$, $j=1, \ldots, b_n$, or $\hat{\kappa}_{n,j}$, $j=1, \ldots, b_n$, respectively. Note that an empirical version of the skewness is given by 
\begin{align*}
 \hat{\gamma}_{n,j}&=\frac{\frac{1}{\ell_n} \sum_{t\in B_{n,j}} \big(X_t-\frac{1}{\ell_n} \sum_{r\in B_{n,j}}X_r\big)^3}{\Big(\frac{1}{\ell_n}\sum_{t\in B_{n,j}} \big(X_t-\frac{1}{\ell_n} \sum_{r\in B_{n,j}}X_r\big)^2\Big)^{3/2} } \\
  & =\frac{\frac{1}{\ell_n}\sum_{t\in B_{n,j}} X_t^3 -3 \Big(\frac{1}{\ell_n}\sum_{t\in B_{n,j}} X_t^2 \Big) \Big(\frac{1}{\ell_n}\sum_{t\in B_{n,j}} X_t \Big) 
  +2 \Big( \frac{1}{\ell_n}\sum_{t\in B_{n,j}} X_t  \Big)^3}{ \Big(\frac{1}{\ell_n}\sum_{t\in B_{n,j}} X_t^2 - \big(\frac{1}{\ell_n}\sum_{t\in B_{n,j}} X_t \big)^2 \Big)^{3/2} },
\end{align*}
which is covered in our setting via the function

\[
 g(x_1,x_2,x_3)= \frac{x_3-3x_1 x_2 +2 x_1^3 }{ \big( x_2-x_1^2  \big)^{3/2}  }.
\]
An empirical version of the kurtosis is given by
\begin{align*}
\hat{\kappa}_{n,j}=& 
\frac{\frac{1}{\ell_n} \sum_{t\in B_{n,j}} \big(X_t-\frac{1}{\ell_n} \sum_{r\in B_{n,j}}X_r\big)^4}{\Big(\frac{1}{\ell_n}\sum_{t\in B_{n,j}} \big(X_t-\frac{1}{\ell_n} \sum_{r\in B_{n,j}}X_r\big)^2\Big)^{2} }\\
=&\bigg(\frac{1}{\ell_n}\sum_{t\in B_{n,j}} X_t^2-\Big(\frac{1}{\ell_n} \sum_{t\in B_{n,j}}X_t\Big)^2\bigg)^{-2}\cdot\bigg(\frac{1}{\ell_n} \sum_{t\in B_{n,j}} X_t^4 -4 \Big(\frac{1}{\ell_n} \sum_{t\in B_{n,j}} X_t   \Big)\Big(\frac{1}{\ell_n} \sum_{t\in B_{n,j}} X_t^3 \Big) \\
& \qquad \qquad \qquad \qquad  \qquad \qquad+6 \Big( \frac{1}{\ell_n} \sum_{t\in B_{n,j}} X_t \Big)^2 \Big( \frac{1}{\ell_n} \sum_{t\in B_{n,j}} X_t^2 \Big) -3 \Big(\frac{1}{\ell_n} \sum_{t\in B_{n,j}} X_t \Big)^4\bigg), 
\end{align*}
which corresponds to the function
\[
  g(x_1,x_2,x_3,x_4) = \frac{x_4-4 x_1 x_3 + 6 x_1^2 x_2 -3 x_1^4  }{(x_2-x_1^2)^2 }. 
\]
\end{enumerate}
\end{example}  

We now state a central limit theorem for U-statistics of this more concrete type of triangular array \eqref{eq: Specific triang array}. The question of the central limit theorem for U-statistics whose entries are Bernoulli shifts has been addressed in 
\citet{MR2060311} and \citet{MR4243516}, but these results do not treat the case of arrays.

\begin{theorem}\label{thm:Bernoulli}

 Assume that the following conditions are satisfied.
 \begin{enumerate}
 \item The function $h$ is Lipschitz-continuous.
 \item $\E{X_1^{2m}}<\infty$ and  
  \begin{equation}\label{eq:condition_on_projectors}
   \sum_{i\in\Z}\sum_{k=1}^m i^2\delta_i\pr{\pr{X_t^k}_{t\in\Z}} <\infty.
  \end{equation}
  \item The function $g$ satisfies
  $g\pr{v_0}=0$, where $v_0=\pr{\E{X_1^k}}_{k=1}^m\in\R^m$. There exists an
$a>0$ such that $g$ is differentiable at each point of $\prod_{k=1}^m
\pr{\E{X_1^k}-2a,\E{X_1^k}+2a}$ and the gradient of $g$ is bounded on $\prod_{k=1}^m
\pr{\E{X_1^k}-2a,\E{X_1^k}+2a}$.  
\item\label{itm:cond_sur_bn_ln_Bernoulli} The sequences $\pr{b_n}_{n\geq 1}$ 
and 
  $\pr{\ell_n}_{n\geq 1}$ go to infinity as $n$ goes 
  to infinity. Moreover, $\lim_{n\to +\infty}
  b_n/\ell_n=0$.
 \end{enumerate}
Let $\eta\colon\R^m\to [0,1]$ be a smooth function with $\eta\pr{x}=1$ 
if $\norm{x-v_0}_2\leq a$
and $\eta\pr{x}=0$ if 
 $\norm{x-v_0}_2>2a$, and define
\begin{equation*}
  W^{\pr{\eta}}_{n,j}:= \sqrt{\ell_n}\pr{g\cdot\eta}\pr{
  \pr{
  \frac 1{\ell_n}\sum_{t\in 
B_{n,j}}X_t^k}_{k=1}^m  }.
 \end{equation*} 
If 
\begin{equation}\label{eq:definition_de_sigma_Bernoulli}
 \sigma^2:=\sum_{t\in\Z}\cov{\sum_{k=1}^m\frac{\partial g}{\partial x_k}\pr{v_0}X_0^k}{\sum_{k=1}^m\frac{\partial g}{\partial x_k}\pr{v_0}X_t^k  }>0,
\end{equation}
 the following convergence in distribution holds
\begin{equation}\label{eq:conv_U_stat_Bernoulli}
 \sqrt{b_n}\pr{U_n
 -\frac 1{b_n\pr{b_n-1}}\sum_{1\leq j\neq k\leq
b_n}\E{h\pr{W^{\pr{\eta}}_{n,j},W^{\pr{\eta}}_{n,k}}
 }    }\to N\pr{0,4\gamma^2},
\end{equation}
where 
 \begin{equation*}
  \gamma^2:= \cov{h\pr{\sigma N,\sigma N'}}{h\pr{\sigma N',\sigma N''}}
 \end{equation*}
for independent standard normally distributed random variables $N$, $N'$ and $N''$.

\end{theorem}
\begin{remark}
Note that the partial derivatives in \eqref{eq:definition_de_sigma_Bernoulli} arise as a consequence of the 
delta method.
\end{remark}

In the examples below, we provide some classes of processes $(X_t)_{t\in \Z}$ that fulfil the above condition \eqref{eq:condition_on_projectors}. For more details, we refer to Section \ref{subsec:Verification examples}. 
\begin{example} \label{Example: Hoelder of linear process}
 Let $X_t$ be a H\"{o}lder continuous function of a linear process as considered, for instance, in  \citet{ MR1487431} and \citet{MR2212691}.
More precisely, define
$X_t$ as 
\begin{equation*}
 X_t:=\varphi\pr{\sum_{j\in\Z}a_j\varepsilon_{t-j}},
\end{equation*}
where $\varphi\colon\R\to\R$ is $\gamma$-H\"{o}lder continuous for some $\gamma\in 
(0,1]$, $\pr{a_j}_{j\in \Z}$ is a sequence of real numbers such that 
$\sum_{j\in\Z}j^2\abs{a_j}^{\gamma}<\infty$ and $\pr{\varepsilon_u}_{u\in\Z}$ is an 
i.i.d.\ sequence such that $\E{\abs{\varepsilon_0}^{2m\gamma}}<\infty$. Then 
$\pr{X_t}_{t\in\Z}$ satisfies condition \eqref{eq:condition_on_projectors}. \end{example}

\begin{example}\label{Example: Functions of Gaussian proc} 
 Assume $X_t$ can be written as a function of a Gaussian linear process: Let $X_t=\varphi(Y_t)$ with
 \begin{equation*}
Y_{t}:=\sum_{ j \in\Z  }a_{j}\varepsilon_{t-j}, 
\end{equation*}
 where $\varphi:\R\rightarrow\R$, 
 $\pr{\varepsilon_{u}}_{u \in \Z}$ is an i.i.d.\ sequence, 
$\varepsilon_{0}$ has 
a standard normal distribution and $a_j \in \R$ for all $j\in\Z$ with $\sum_{ j \in\Z  
}\abs{a_{j}}<\infty$ as well as $\sum_{ j \in\Z  
}a_{j}^2=1$.  Such processes were considered, e.g., in \citet{MR2498953}.
Given  that $\varphi\pr{Y_0}\in\mathbb L^2$ 
and 
$\E{\varphi\pr{Y_{0}}}=0$, the following expansion holds: 
\begin{equation*}
\varphi\pr{Y_{t}}=\sum_{q=1}^{\infty}c_q\pr{\varphi}H_q\pr{Y_{t}},
\end{equation*}
where the
$q$-th Hermite polynomial is defined by 
\begin{equation*}
H_q\pr{x}:=\pr{-1}^q\exp\pr{\frac{x^2}2}\frac{d^q}{dx^q}\exp\pr{-\frac{x^2}2}
\end{equation*}
and
\begin{equation*}
c_q\pr{\varphi}:=\frac{1}{q!}\E{\varphi\pr{Y_{0}}H_q\pr{Y_{0}}},
\end{equation*}
provided that $\sum_{q=1}^{\infty}q!c_q\pr{\varphi}^2$ converges. 
 Then condition 
\eqref{eq:condition_on_projectors} is met if $\sum_{j\in\Z}j^2\abs{a_j}<\infty$ and

\begin{equation*}
\sum_{k=1}^m\sum_{q=1}^{\infty}\sqrt{q\cdot q!} 
\abs{c_q\pr{\varphi^k-\E{\varphi^k\pr{Y_0}}}}<\infty.
\end{equation*}

\end{example}

\begin{example} \label{Example: Volterra}
Let $\pr{X_t}_{t\in \Z}$ be a Volterra process, i.e. let
\begin{equation*}
 X_t:=\sum_{j,j'\in\Z,j\neq j'}a_{j,j'}\varepsilon_{t-j}\varepsilon_{t-j'},
\end{equation*}
where $\pr{\varepsilon_u}_{u\in\Z}$ is i.i.d.\ centred, $\E{\varepsilon_0^2}<\infty$ 
and $\sum_{j,j'\in\Z,j\neq j'}a_{j,j'}^2<\infty$. Such processes and extensions to Volterra series are considered in \citet{MR713994} and
 \citet{MR991969}.
If $\E{\varepsilon_0^{2m}}<\infty$ as well as $$\sum_{j\in\Z} j^2\sqrt{\sum_{j'\in\Z, j'\neq j} \pr{a_{j,j'}^2+a_{j',j}^2}}  <\infty,$$
then $\pr{X_t}_{t\in\Z}$ satisfies condition
\eqref{eq:condition_on_projectors}.
\end{example}

It would be more natural to centre $U_n$ in \eqref{eq:conv_U_stat_Bernoulli} by $\sum_{1\leq j\neq k\leq b_n}\E{ h\pr{W_{n,j},W_{n,k}}} $ rather than by the truncated version $\sum_{1\leq j\neq k\leq b_n}\E{ h\pr{W_{n,j}^{\pr{\eta}},W_{n,k}^{\pr{\eta}}}}$. However, the conditions of Theorem \ref{thm:Bernoulli} do not guarantee that  $\E{ \abs{h\pr{W_{n,j},W_{n,k}}}}$ exists, as the following example shows.
\begin{example}
Consider the case $m=2$, $g\pr{x_1,x_2}=\log x_2 \1{x_2>0}$ and 
$h\pr{x,y}=\abs{x-y}$ with i.i.d.\ observations $(X_t)_{t\in\Z}$. 
Since $W_{n,1}$ and $W_{n,2}$ are  consequently likewise independent and identically distributed, finiteness of 
$\E{\abs{W_{n,1}-W_{n,2}  }}$ is 
equivalent to the finiteness of $\E{\abs{W_{n,1} }}$.
 Now, it suffices to find an i.i.d.\ sequence $\pr{X_t}_{t\in \Z}$ such that $\E{X_1^4}<\infty$ and 
\begin{equation*}
 \E{\abs{\log\pr{\frac{1}{\ell_n}\sum_{t=1}^{\ell_n} X_t^2}  }}=\infty.
\end{equation*}
By choosing the distribution of $X_1$ as
\begin{equation*}
 \PP\pr{X_1^2=\exp\pr{-\exp\pr{\exp\pr{k}}} }=2^{-k} 
\end{equation*}
for $k\geq 1$, it follows
\begin{align*}
\abs{\log\pr{\frac{1}{\ell_n}\sum_{t=1}^{\ell_n} X_t^2}  }
 \geq & \sum_{k\geq 1} 
 \abs{\log\pr{\frac{1}{\ell_n}\sum_{t=1}^{\ell_n}\exp\pr{-\exp\pr{\exp\pr{k}}}}  }
 \1{\bigcap_{t=1}^{\ell_n} \ens{X_t^2=\exp\pr{-\exp\pr{\exp\pr{k}}}} }\\
 =& \sum_{k\geq 1}\exp\pr{\exp{k}} \1{\bigcap_{t=1}^{\ell_n} 
\ens{X_t^2=\exp\pr{-\exp\pr{\exp\pr{k}}}} }.
\end{align*}
Taking expectations and using independence leads to 
\begin{equation*}
 \E{\abs{\log\pr{\frac{1}{\ell_n}\sum_{t=1}^{\ell_n} X_t^2}  } }
 \geq \sum_{k\geq 1}\exp\pr{\exp (k)}2^{-k\ell_n}=\infty.
\end{equation*}
\end{example}

By imposing additional assumptions on the function $g$,  the dependence coefficients as well as the sequences
$\pr{b_n}_{n\geq 1}$ and $\pr{\ell_n}_{n\geq 1}$, we are able to replace the centring term in \eqref{eq:conv_U_stat_Bernoulli} computed from the $W_{n,j}^{(\eta)}$'s by an expression that does not require truncation.  
 
\begin{proposition} \label{Prop: Replacing centring Bernoulli}
Assume that the assumptions of Theorem \ref{thm:Bernoulli} hold. 
If additionally all the second order partial derivatives of $g$ at $v_0$ exist and if there exists a $\kappa\in\pr{0,1}$ such that $\sum_{i\in\Z}\sum_{k=1}^m \abs{i}^{\frac 12+\frac{1}{2\kappa}}\delta_i\pr{\pr{X_t^k}_{t\in\Z}} <\infty$ and $b_n/\ell_n^{1-\kappa}\rightarrow 0$, then the following convergence in distribution holds
\begin{equation*}
 \sqrt{b_n}\pr{U_n
 - \E{h\pr{ Z_n,Z'_n}
 }    }\to N\pr{0,4\gamma^2},
\end{equation*}
where
\begin{equation*}
 Z_n:=\frac{1}{\sqrt{\ell_n}}\sum_{k=1}^m\frac{\partial g}{\partial x_k}\pr{v_0}\sum_{t=1}^{\ell_n}\pr{X_t^k-\E{X_1^k}}
 \end{equation*}
and $Z'_n$ is an independent copy of $Z_n$.
 \end{proposition}

\begin{remark}
The above proposition introduces a centring term which is easier to handle than the original one in Theorem~\ref{thm:Bernoulli}. However, if we want to use $U_n$ as a test statistic for testing for a constant value of the parameter $g\pr{\E{X_t},\E{X_t^2},\ldots,\E{X_t^m}}$, we need to be able to explicitly calculate the centring term. This is achieved by the following corollary, where we show that, for the important example when $h(x,y)=\abs{x-y}$, we can replace $\E{\abs{Z_n-Z_n^\prime}}$ by $\sigma\, \E{\abs{ N - N^\prime}} = 2\,\sigma/\sqrt{\pi}$. The remaining parameter $\sigma$ can be estimated by standard procedures for estimating long-run variances.
\end{remark}
\begin{corollary}\label{cor: replacing centring term bernoulli}
Suppose that the time series $\pr{X_t}_{t\in\Z}$ can be written as a one-sided Bernoulli shift, that is, $X_t= f\pr{\pr{\varepsilon_{t-u}}_{u\geq 0}  }$. We assume moreover that there exists a $2<p\leq 3$ such that
\begin{enumerate}
\item  $\E{\abs{X_t}^{p\cdot m}}<\infty$ and
\begin{equation}
\sum_{i\geq 0}\sum_{k=1}^m \pr{i^2\delta_{i,p}\pr{(X_t^k)_{t\in\Z}}+i^{5/2}\delta_{i,2}\pr{(X_t^k)_{t\in\Z}} }<\infty,
\end{equation} where $
\delta_{i,p}((X_t)_{t\in\Z}):=\norm{X_0-X_0^{*,i}}_p$, and 
\item $ \sqrt{b_n}\ell_n^{1-p/2} \pr{\log\ell_n}^{p/2}+ \sqrt{b_n/\ell_n}\to 0$.
\end{enumerate}
Denote
\begin{equation*}
U_n:=\frac{1}{b_n(b_n-1)} \sum_{1\leq j \neq k\leq b_n} \abs{W_{n,j} -W_{n,k}}, 
\end{equation*} 
where $W_{n,j}$ is defined as in \eqref{eq: Specific triang array}.
Then the following convergence in distribution holds
\[
  \sqrt{b_n} \big(U_n- \frac{2\sigma}{\sqrt{\pi}}\big) \longrightarrow N(0,4\gamma^2),
\]
where $\sigma$ and $\gamma^2$ are defined as in Theorem~\ref{thm:Bernoulli}.
\end{corollary}

\section{Proofs}\label{sec:proofs}
\subsection{Proof of Theorem \ref{thm:TLC_triangular_array}}

 We use the Hoeffding decomposition of the kernel function $h$ and define 
 \begin{align*}  
 \theta_n&:=\E{h\pr{W_{n,1},W_{n,2}}},\\
 h_{1,n}(x)&:=\E{h(x,W_{n,2})}-\theta_n,\\
h_{2,n}\pr{x,y}&:=h\pr{x,y}-h_{1,n}\pr{x}-h_{1,n}\pr{y}-\theta_n.
 \end{align*}
 At  the level of the U-statistic, we then obtain
\begin{align*}
U(n)-\E{U(n)}
=&\frac{1}{b_n\pr{b_n-1}}\sum_{1\leq j \neq k\leq b_n}h\pr{W_{n,j},W_{n,k}}
 -\theta_n \\
=&\frac{2}{b_n}\sum_{j=1}^{b_n}h_{1,n}\pr{W_{n,j}}+\frac{1}{b_n\pr{b_n-1}}\sum_{
1\leq j\neq k\leq b_n}h_{2,n}\pr{W_{n,j},W_{n,k}}.
\end{align*}
In the following, we will prove the convergence in distribution 
\begin{equation}\label{eq:TLC_triangular_array_linear_part}
 \frac{2}{\sqrt{b_n} 
\gamma_n}\sum_{j=1}^{b_n}h_{1,n}\pr{W_{n,j}} \to 
N\pr{0,4},
\end{equation}
and the convergence in probability
\begin{equation}\label{eq:TLC_triangular_array_degenerated_part}
 \frac{1}{b_n^{3/2}\gamma_n}\sum_{
1\leq j\neq k\leq b_n}h_{2,n}\pr{W_{n,j},W_{n,k}} \to   0, 
\end{equation}
where $\gamma_n^2=\cov{h\pr{W_{n,1},W_{n,2}} }{h\pr{W_{n,2},W_{n,3}}}$. 
The assertion then follows by an application of Slutzky's Lemma.

Starting with \eqref{eq:TLC_triangular_array_linear_part}, we will  apply Lindeberg's 
central limit theorem to the triangular array 
$Y_{n,j}:=h_{1,n}\pr{W_{n,j}}/\pr{\gamma_n \sqrt{b_n}}.$
Note that by construction, the $Y_{n,j}$'s are identically distributed with $\E{Y_{n,1}}=0$. Moreover, it holds $\var{Y_{n,1}}=1/b_n$ since by independence
\begin{align*}
\var{h_{1,n}\pr{W_{n,1}}}
&=\var{\E{h\pr{W_{n,1},W_{n,2}  } \mid W_{n,1} }}\\
  & =\E{\E{h\pr{W_{n,1},W_{n,2}  } \mid W_{n,1} }^2}- \pr{ \E{ \E{h\pr{W_{n,1},W_{n,2}  } \mid W_{n,1}  }  }}^2\\
  & =\int_{\R}\E{h\pr{x,W_{n,2}}h\pr{x,W_{n,3}} }d\PP_{W_{n,1}}\pr{x}- \pr{ \E{ h\pr{W_{n,1},W_{n,2}  }}  }^2\\
& =\E{h\pr{W_{n,1},W_{n,2}  } \cdot h\pr{W_{n,1},W_{n,3}  }}- \pr{ \E{ h\pr{W_{n,1},W_{n,2}  } }}^2\\
 &= \cov{h\pr{W_{n,1},W_{n,2}} }{h\pr{W_{n,2},W_{n,3}}}= \gamma_n^2.
\end{align*}
It thus remains to verify the Lindeberg 
condition, that is, to show that for all $\varepsilon>0$, 
\begin{equation*}
 \sum_{j=1}^{b_n}\E{Y_{n,j}^2\mathbf{1}_{\ens{\abs{Y_{n,j}}>\varepsilon  }}}\to 0.
\end{equation*}
Since the random variables $Y_{n,j}, 1\leq j\leq b_n$, are identically 
distributed, Lindeberg's condition reduces to 
\begin{align*}
 \sum_{j=1}^{b_n}\E{Y_{n,j}^2\mathbf{1}_{\ens{\abs{Y_{n,j}}>\varepsilon  }}}=
 b_n \E{Y_{n,1}^2\mathbf{1}_{\ens{\abs{Y_{n,1}}>\varepsilon  }}}
 =\E{\frac{h^2_{1,n}\pr{W_{n,1}  }  
}{\gamma_n^2}\1{\ens{\frac{\abs{h_{1,n}\pr{W_{n,1}  }  }     }{\gamma_n}  
>\varepsilon  \sqrt{b_n}    }} } \rightarrow 0.
\end{align*}
Observe that by property \eqref{eq:condition_on_h} of the kernel function $h$, we have  $\abs{h_{1,n}\pr{x}}
\leq 3C\E{\abs{W_{n,1}}}+C\abs{x}$ and hence $h^2_{1,n}\pr{x}\leq 
18C^2\E{ W^2_{n,1}}+2C^2x^2$, from which it follows that
\begin{equation*}
 \frac{h^2_{1,n}\pr{W_{n,1}  }  
}{\gamma_n^2}\leq18C^2+2C^2\frac{W_{n,1}^2}{\gamma_n^2}.
\end{equation*}
Consequently, the uniform integrability of the  sequence 
$\pr{h^2_{1,n}\pr{W_{n,1}  }  
/\gamma_n^2}_{n\geq 1}$ follows from that of the sequence   
$\pr{W_{n,1}^2/\gamma_n^2}_{n\geq 1}$, and thus Lindeberg's
condition is met.

It remains to verify \eqref{eq:TLC_triangular_array_degenerated_part}. Since 
$\E{h_{2,n}\pr{W_{n,j},W_{n,k}  }h_{2,n}\pr{W_{n,j'},W_{n,k'}  }  }=0$ if 
$\{j,k\}\neq \{j',k'\}$, we obtain 
\begin{equation}\label{eq:moment_ordre_deux_partie_deg}
 \E{\pr{ \frac{1}{b_n^{3/2}\gamma_n}\sum_{
1\leq j\neq k\leq b_n}h_{2,n}\pr{W_{n,j},W_{n,k}}  }^2}
=\frac{2\pr{b_n-1}}{\gamma_n^2b_n^2}\E{h^2_{2,n}\pr{W_{n,1},W_{n,2}}  }.
\end{equation}
By the properties of the kernel function $h$ and the definition of $h_{2,n}$,
there exists a constant $C$ independent of $n$, $x$ and $y$ such that 
$\abs{h_{2,n}\pr{x,y}}\leq C\pr{1+\abs{x}+\abs{y}+\E{\abs{W_{n,1}}}}$ and hence 
$\E{h^2_{2,n}\pr{W_{n,1},W_{n,2}}  }\leq C'\E{W_{n,1}^2}$ for a constant $C'$ 
independent of $n$. Combining this with 
\eqref{eq:moment_ordre_deux_partie_deg}, we get that 
\begin{equation*}
 \E{\pr{ \frac{1}{b_n^{3/2}\gamma_n}\sum_{
1\leq j\neq k\leq b_n}h_{2,n}\pr{W_{n,j},W_{n,k}}  }^2}
\leq \frac{2C}{b_n}\sup_{n\geq 1}\E{\frac{W_{n,1}^2}{\gamma_n^2}}
\end{equation*}
and the uniform integrability of $\pr{W_{n,1}^2/\gamma_n^2}_{n\geq 1}$ guarantees the
finiteness of the above supremum. This concludes the proof of 
\eqref{eq:TLC_triangular_array_degenerated_part} and that of 
Theorem~\ref{thm:TLC_triangular_array}.
\subsection{Sketch of proof for Theorem~\ref{thm:Bernoulli}} 
This section outlines the proof ideas for Theorem~\ref{thm:Bernoulli}, while all details can be found in the next section. 
First, we reduce the problem to 
 the case where in $U_n$, the term $W_{n,j}$ is 
 replaced by $W^{\pr{\eta}}_{n,j}=\sqrt{\ell_n}\pr{g\cdot\eta}\big(  \big(  \frac 1{\ell_n}\sum_{t\in B_{n,j}}X_t^k\big)_{k=1}^m \big)$. 
We thus define 
\begin{equation*}
U^{\pr{\eta}}_n:=\frac 1{b_n\pr{b_n-1}}\sum_{1\leq j\neq k\leq 
b_n}h\pr{W^{\pr{\eta}}_{n,j},W^{\pr{\eta}}_{n,k}}.
\end{equation*} 
The next lemma shows that we can replace $U_n$ in the central limit theorem by $U^{\pr{\eta}}_n$.

 \begin{lemma}\label{Lemma: Approx U(a) Bernoulli}
Let the assumptions of Theorem~\ref{thm:Bernoulli} hold. Then $\PP\pr{U_n\neq U^{\pr{\eta}}_n}\to 0$. In particular, 
$ \pr{\sqrt{b_n}\pr{U_n-U_n^{\pr{\eta}}}}_{n\geq 1}$ converges in 
probability to zero.
 \end{lemma}
It thus suffices to prove the convergence in distribution
\begin{equation}\label{eq:conv_U_stat_Bernoulli_bis}
\sqrt{b_n}\pr{U_n^{\pr{\eta}} 
 -\frac 1{b_n\pr{b_n-1}}\sum_{1\leq j\neq k\leq 
b_n}\E{h\pr{W^{\pr{\eta}}_{n,j},W^{\pr{\eta}}_{n,k}} 
 }    }\to N\pr{0,4\gamma^2}.
\end{equation}
 To do so, we use a second approximation step and replace the $W^{\pr{\eta}}_{n,j}$'s by 
\begin{equation*}
 W_{n,j}^{\pr{M}}:= \sqrt{\ell_n}\pr{g\cdot \eta}
 \pr{ \pr{ \frac{1}{\ell_n}\sum_{t\in B_{n,j}} 
 \E{X_t^k\mid \Fca_{t-M}^{t+M}  } }_{k=1}^m
 },
\end{equation*}
where $\Fca_{M}^N:=\sigma\pr{\varepsilon_u,M\leq u\leq N}$ for $M,N\in\Z$ with $ M\leq N$. The random variables $X_t$ are thus replaced by random variables 
depending only on those $\varepsilon_{t-u}$ with $\abs{t-u}\leq M$. Note that this way, the entries of the U-statistic become almost independent (up to some small overlap in consecutive blocks). We define
$$U_n^{(M)}:= \frac 1{b_n\pr{b_n-1}}
 \sum_{1\leq j\neq k\leq b_n} h\pr{W^{\pr{M}}_{n,j},W^{\pr{M}}_{n,k}}.$$
We can now decompose the expression on the left hand side of \eqref{eq:conv_U_stat_Bernoulli_bis} for each fixed $M\geq 1$ via
\begin{equation}\label{eq:decomposition_UnM}
\sqrt{b_n}(U_n^{(M)}-\E{U_n^{(M)}})+R_{n,M},
\end{equation}
where the remainder term is given by the telescoping sum
\begin{align*}
 R_{n,M}:= \frac{1}{\sqrt{b_n}\pr{b_n-1}}\sum_{N\geq M}
\sum_{1\leq j\neq k\leq b_n}
 &\Big( h\pr{W^{\pr{N+1}}_{n,j},W^{\pr{N+1}}_{n,k}}-\E{h\pr{W^{\pr{N+1}}_{n,j},W^{\pr{N+1}}_{n,k} }  } \\
& - \Big(h\pr{W^{\pr{N}}_{n,j},W^{\pr{N}}_{n,k}}-\E{h\pr{W^{\pr{N}}_{n,j},W^{\pr{N}}_{n,k} }  }\Big)\Big).
\end{align*}
The following three lemmas show that the first term in \eqref{eq:decomposition_UnM} converges to the desired normal distribution, while the continuity of $g\cdot\eta$ will guarantee that the remainder term $R_{n,M}$ becomes asymptotically negligible.

\begin{lemma}\label{lem:conv_m_dep}
Let the assumptions of Theorem~\ref{thm:Bernoulli} hold. Then,  there exists an $M_0\in\N$ such that for all fixed $M\geq M_0$, the sequence 
 $\pr{\sqrt{b_n}(U_n^{(M)}-\mathbb{E}[U_n^{(M)}])}_{n\geq 1}$ converges in distribution to a centred normally distributed random variable with variance 
  \begin{equation*}
 4 \gamma_M^2:= 4\cov{ h\pr{\sigma_M N,\sigma_M N'}  } {h\pr{\sigma_M N,\sigma_M 
N''}},
 \end{equation*}
where $N$, $N'$ and $N''$ are three independent standard normal random 
variables and 
\begin{equation*}
 \sigma_M^2:=\sum_{t=-M-1}^{M+1}
  \cov{\E{\sum_{k=1}^m\frac{\partial g}{\partial x_k}\pr{v_0}X_0^k\Big| \Fca_{-M}^M  }}{\E{\sum_{k=1}^m\frac{\partial g}{\partial x_k}\pr{v_0}X_t^k\Big| \Fca_{t-M}^{t+M}
 }}.
\end{equation*}
\end{lemma}
The central ingredient in the proof of Lemma \ref{lem:conv_m_dep} is Theorem \ref{thm:TLC_triangular_array}. To meet its conditions, we approximate $U_n^{(M)}$ by yet another U-statistic, which has independent entries (details are given in the next section).

\begin{lemma}\label{lem:conv_of_N_M}
Let the assumptions of Theorem~\ref{thm:Bernoulli} hold. Then the sequence $\pr{4\gamma_M^2}_{M\geq 1}$ converges  to  $4\gamma^2=4\cov{h\pr{\sigma N,\sigma N'}}{h\pr{\sigma N',\sigma N''}}$.
\end{lemma}

\begin{lemma}\label{lem:approximation_by_m_dep}
Let the assumptions of Theorem~\ref{thm:Bernoulli} hold. Then, for each $\varepsilon>0$, it holds 
 \begin{equation*}
  \lim_{M\to \infty}\limsup_{n\to \infty}
  \PP(\abs{R_{n,M}}>\varepsilon  )=0.
 \end{equation*}
\end{lemma}  

In the final step, we apply Theorem~4.2 of \citet{MR0233396}. This theorem states that for stochastic processes $\pr{Y_{m,n}}_{m,n\geq 1}$, $\pr{Y'_n}_{n\geq 1}$, $\pr{Z_m}_{m\geq 1}$, and a random variable $Z$, satisfying
\begin{align*}
& Y_{m,n} \longrightarrow Z_m \mbox{ in distribution, as }    n\rightarrow \infty, \mbox{ for each } m, \\
& Z_m \longrightarrow Z \mbox{ in distribution, as }  m\rightarrow \infty,\\
& \lim_{m\to\infty}\limsup_{n\to\infty}\PP\pr{\abs{Y'_n-Y_{m,n}}>\varepsilon}=0, \mbox{ for all } \varepsilon >0, 
\end{align*}
we may conclude that $Y'_n\to Z$ in distribution.  
A combination of
Lemmas~\ref{lem:conv_m_dep}, \ref{lem:conv_of_N_M} and 
\ref{lem:approximation_by_m_dep} thus yields the desired convergence in \eqref{eq:conv_U_stat_Bernoulli_bis} and thus finishes the proof.

\subsection{Proof details} 
\label{Subsec: Proof details}
\begin{proof}[Proof of Lemma \ref{Lemma: Approx U(a) Bernoulli}]

We start by noting the following inclusions
\begin{equation*}
 \ens{U_n\neq U^{\pr{\eta}}_n}\subset \bigcup_{j=1}^{b_n}
 \ens{W_{n,j}\neq W^{\pr{\eta}}_{n,j}} 
 \subset \bigcup_{j=1}^{b_n}\bigcup_{k=1}^m
 \ens{\abs{ \pr{\frac{1}{\ell_n}
   \sum_{t\in B_{n,j}}X_t^k} -\E{X_1^k}}  >2a/m  }.
 \end{equation*}
Hence, using the fact that for each $k$, the 
random variables $\pr{\frac{1}{\ell_n}
   \sum_{t\in B_{n,j}}X_t^k -\E{X_1^k}}$, $j=1, \ldots, b_n,$  
have the same distribution, and by an application of
Chebychev's inequality, we obtain  
\begin{equation*}
 \PP (U_n\neq U^{\pr{\eta}}_n)\leq m^2\frac{b_n}{4a^2\ell_n }\sum_{k=1}^m
  \var{ \frac{1}{\sqrt{\ell_n}}
   \sum_{t= 1}^{\ell_n}X_t^k}.
\end{equation*}
By Lemma~\ref{lem:moments_stationary_sequence}, it holds $\var{ \frac{1}{\sqrt{\ell_n}}
   \sum_{t= 1}^{\ell_n}X_t^k}\leq \big(\sum_{i\in\Z} \delta_i((X_t^k)_{t\in\Z})\big)^2$ and thus, we have $\PP(U_n\neq U^{\pr{\eta}}_n)\leq Cb_n/\ell_n$, which converges to zero by assumption. 
\end{proof}

\begin{proof}[Proof of Lemma~\ref{lem:conv_m_dep}]
 Let $M\geq 1$ be fixed. We intend to approximate $U_n^{(M)}$ once more to obtain a U-statistic $\widetilde{U}_n^{(M)}$ with independent entries. Thus, we define
\begin{equation*}
 \widetilde{W}_{n,j}^{\pr{M}}:= \sqrt{\ell_n}\pr{g\cdot \eta}
 \pr{ \pr{\frac{1}{\ell_n}\sum_{t=\pr{j-1}\ell_n+M+1}^{j\ell_n
 -M-1} 
 \E{X_t^k\mid \Fca_{t-M}^{t+M} } }_{k=1}^m
 }
\end{equation*} 
 for $j=1, \ldots, b_n$, assuming that $n$ is large enough such that $2M\leq \ell_n$. Due to the shortening of the blocks $B_{n,j}$ by $M$ observations on each side, the $ \widetilde{W}_{n,j}^{\pr{M}}$'s are independent. Next, we define 
 $$\widetilde{U}_n^{(M)}: =   \frac 1{b_n\pr{b_n-1}}
 \sum_{1\leq j\neq k\leq b_n} h\pr{\widetilde{W}^{\pr{M}}_{n,j},\widetilde{W}^{\pr{M}}_{n,k}}.$$
Since the function 
$g\cdot \eta$ is Lipschitz-continuous, it holds
$$
 \abs{W_{n,j}^{\pr{M}}-\widetilde{W}_{n,j}^{\pr{M}}}
 \leq C\sum_{k=1}^m\frac{1}{\sqrt{\ell_n}}
 \sum_{t=\pr{j-1}\ell_n+1}^{\pr{j-1}\ell_n+M}
  \E{\abs{X_t^k}\Big|  \Fca_{t-M}^{t+M}}\\
  + C\sum_{k=1}^m\frac{1}{\sqrt{\ell_n}}
 \sum_{t=j\ell_n-M }^{j\ell_n}
  \E{\abs{X_t^k}\Big|  \Fca_{t-M}^{t+M}}.
$$
 Taking expectations, we obtain 
 $$
 \E{\abs{W_{n,j}^{\pr{M}}-\widetilde{W}_{n,j}^{\pr{M}}}}
 \leq C \frac{M}{\sqrt{\ell_n}} $$
since $\E{X_1^{m}}<\infty$ by assumption. Thus, for the U-statistic, we have
\begin{align*}
&\sqrt{b_n}\E{\abs{ \pr{U_n^{(M)}-\E{U_n^{(M)}}} - \Big(\widetilde{U}_n^{(M)}-\E{\widetilde{U}_n^{(M)}}\Big)}}\\
\leq & 2  \frac{\sqrt{b_n}}{b_n\pr{b_n-1}}
 \sum_{1\leq j\neq k\leq b_n} \E{\abs{h\pr{{W}^{\pr{M}}_{n,j},{W}^{\pr{M}}_{n,k}}-h\pr{\widetilde{W}^{\pr{M}}_{n,j},\widetilde{W}^{\pr{M}}_{n,k}}}} \\
 \leq & C \sqrt{b_n} \E{\abs{W_{n,1}^{\pr{M}}-\widetilde{W}_{n,1}^{\pr{M}}}} \leq CM \frac{\sqrt{b_n}}{\sqrt{\ell_n}},
\end{align*}
where the last but one inequality follows by the Lipschitz-continuity of the kernel $h$ and the stationarity of the ${W}^{\pr{M}}_{n,j}$'s and of the $\widetilde{W}^{\pr{M}}_{n,j}$'s.

In the following, it hence suffices to prove that $\sqrt{b_n}\Big(\widetilde{U}_n^{(M)}-\E{\widetilde{U}_n^{(M)}}\Big)$ converges in distribution to a centred normal random variable $N_M$ with variance $4\gamma_M^2$. 
We first show that 
\begin{equation*}
 \gamma_{M,n}^2:= 
 \cov{h\pr{  \widetilde{W}_{n,1}^{\pr{M}}    ,
   \widetilde{W}_{n,2}^{\pr{M}}}}{h\pr{  \widetilde{W}_{n,1}^{\pr{M}}    ,
   \widetilde{W}_{n,3}^{\pr{M}}}} \rightarrow \gamma_M^2
\end{equation*}
and afterwards the convergence in distribution
 $$\frac{\sqrt{b_n}}{\gamma_{M,n}}\Big(\widetilde{U}_n^{(M)}-\E{\widetilde{U}_n^{(M)}}\Big)\to  N({0}, {4}) ,$$
which combined yield the assertion.

We start by verifying $\gamma_{M,n}^2\rightarrow \gamma_M^2$ by means of Lemma \ref{lem:convergence_of_sigma_n}. To apply Lemma \ref{lem:convergence_of_sigma_n}, we have to check that the sequence $\big(\big(\widetilde{W}_{n,1}^{\pr{M}}\big)^2\big)_{n\geq 1}$ is uniformly integrable.
  By the assumptions on  $g$ and $\eta$,  the gradient of $g\cdot \eta$ is uniformly bounded over $\R^m$, such that
    \begin{align*}
  \pr{\widetilde{W}_{n,1}^{\pr{M}}}^2= &  \ell_n
  \pr{g^2 \cdot \eta^2}\pr{v_0+\pr{\frac{1}{\ell_n}\sum_{t= M+1}^{ \ell_n -M-1}
 \E{X_t^k\mid \Fca_{t-M}^{t+M} }  }_{k=1}^m -v_0 }\\
  \leq &\ell_n C^2 \sum_{k=1}^m \pr{\frac{1}{\ell_n}\sum_{t= M+1}^{ \ell_n -M-1}
 \E{X_t^k\mid \Fca_{t-M}^{t+M} }-\E{X_1^k  }}^2
  \end{align*}
and uniform integrability follows from Lemma~\ref{lem:conver_variance_UI}.
We additionally have to show that $\widetilde{W}_{n,1}^{\pr{M}}\to N\pr{0,\sigma_M^2}$ in distribution to apply Lemma \ref{lem:convergence_of_sigma_n}. To do so, we will use  the differentiability of $g\cdot\eta$ at $v_0$ and the fact that $\pr{g\cdot\eta}\pr{v_0}=0$ in order to write
 \begin{equation*}\label{eq:Taylor_g_eta}
  \pr{g\cdot \eta}\pr{v_0+z}=\sum_{k=1}^mz_k\frac{\partial g}{\partial x_k}\pr{v_0}+\varepsilon\pr{z},
 \end{equation*}
where $\varepsilon\pr{z}/\norm{z}_2\to 0$ as $\norm{z}_2\to 0$. Setting $z=\pr{\frac{1}{\ell_n}\sum_{t=M+1}^{\ell_n-M-1}\big(\E{X_t^k\mid\F_{t-M}^{t+M}}-\E{X_1^k}\big) }_{k=1}^m$, we thus obtain
\begin{align*}\label{eq:taylor_expansion_Wn1}
 \widetilde{W}_{n,1}^{\pr{M}}= & \frac{1}{\sqrt{\ell_n}}\sum_{t=M+1}^{\ell_n-M-1}\sum_{k=1}^m
 \frac{\partial g}{\partial x_k}\pr{v_0}\big(\E{X_t^k\mid\F_{t-M}^{t+M}}-\E{X_1^k}\big) \\
& +\sqrt{\ell_n}\varepsilon\pr{ \pr{\frac{1}{\ell_n}\sum_{t=M+1}^{\ell_n-M-1}\big(\E{X_t^k\mid\F_{t-M}^{t+M}}-\E{X_1^k} \big) }_{k=1}^m  }.
\end{align*}
By the central limit theorem (see Lemma \ref{lem:TLC_suite_bernoullienne} in the appendix), the first of the above terms converges to a centred normal distribution with variance $\sigma_M^2$.
Moreover, due to the properties of the function $\varepsilon$ and the strong law of large 
numbers, the second of the above terms converges in probability to $0$. Indeed, 
\begin{align*}
&\sqrt{\ell_n}\varepsilon\pr{ \pr{\frac{1}{\ell_n}\sum_{t=M+1}^{\ell_n-M-1}\big(\E{X_t^k\mid\F_{t-M}^{t+M}}-\E{X_1^k}\big)  }_{k=1}^m   }\\
= & \frac{\varepsilon\pr{
\pr{\frac{1}{\ell_n}\sum_{t=M+1}^{\ell_n-M-1}\big(\E{X_t^k\mid\F_{t-M}^{t+M}}-\E{X_1^k}\big) }_{k=1}^m
}}{\sqrt{\sum_{k=1}^m \frac{1}{\ell_n}
\pr{\sum_{t=M+1}^{\ell_n-M-1}\big(\E{X_t^k\mid\F_{t-M}^{t+M}}-\E{X_1^k}\big)}^2  }  }\\
&\cdot \sqrt{\sum_{k=1}^m \frac{1}{\ell_n}
\pr{\sum_{t=M+1}^{\ell_n-M-1}\big(\E{X_t^k\mid\F_{t-M}^{t+M}}-\E{X_1^k}\big)}^2  },
\end{align*}
which is a product of the form $A_n B_n$, where $\pr{A_n}_{n\geq 1}$ 
converges in probability  to $0$ and $\pr{B_n}_{n\geq 1}$ is tight.
We hence have $\widetilde{W}_{n,1}^{\pr{M}}\to N\pr{0,\sigma^2_M}$ in distribution  and Lemma \ref{lem:convergence_of_sigma_n} implies $\gamma_{M,n}^2\rightarrow \gamma_M^2$. 

Turning towards showing
$\sqrt{b_n}\big(\widetilde{U}_n^{(M)}-\mathbb{E}[\widetilde{U}_n^{(M)}]\big)/\gamma_{M,n}\rightarrow N({0}, {4})$ in distribution, we have to check the assumptions of Theorem \ref{thm:TLC_triangular_array} and  therefore have to prove the uniform integrability of $((\widetilde{W}_{n,1}^{\pr{M}}/\gamma_{M,n})^2)_{n\geq 1}$. The uniform integrability of $\big(\big(\widetilde{W}_{n,1}^{\pr{M}}\big)^2\big)_{n\geq 1}$ has already been shown above. Moreover, since we will show  in the proof of Lemma~\ref{lem:conv_of_N_M} below that $\gamma_M^2\to \gamma^2>0$, there exists an $M_0$ such that for all $M\geq M_0$, it holds $\gamma_M^2>0$. Since for fixed $M\geq M_0$, we have shown above that $\gamma_{M,n}^2\rightarrow\gamma_M^2$, there exists an $n_{0,M}$ such that $\gamma_{M,n}^2 \geq \gamma_M^2/2>0$ for all $n\geq n_{0,M}$. Hence, the sequence $(\gamma_{M,n}^{-2})_{n\geq n_{0,M}}$ is bounded and $((\widetilde{W}_{n,1}^{\pr{M}}/\gamma_{M,n})^2)_{n\geq n_{0,M}}$ is uniformly integrable, such that Theorem \ref{thm:TLC_triangular_array} yields the desired convergence.
\end{proof}

\begin{proof}[Proof of Lemma~\ref{lem:conv_of_N_M}]
It suffices to prove that $\gamma_M^2\to \gamma^2$. Since $h$ is 
Lipschitz-continuous, the mapping $x\mapsto
\cov{ h\pr{xN,xN'}}{h\pr{xN,xN''}}$ is continuous,  
where $N$, $N'$ and $N''$ denote three independent standard normal random 
variables. 
Therefore, we only need to check that $\sigma_M^2\to \sigma^2$, where
\begin{equation*}
 \sigma_M^2=\sum_{t=-2M-1}^{2M+1}
  \cov{\E{\sum_{k=1}^m\frac{\partial g}{\partial x_k}\pr{v_0}X_0^k\Big| \Fca_{-M}^M  }}{\E{\sum_{k=1}^m\frac{\partial g}{\partial x_k}\pr{v_0}X_t^k\Big| \Fca_{t-M}^{t+M}}}
\end{equation*}
and 
\begin{equation*}
 \sigma^2=\sum_{t\in\Z}\cov{\sum_{k=1}^m\frac{\partial g}{\partial x_k}\pr{v_0}X_0^k }{ \sum_{k=1}^m\frac{\partial g}{\partial x_k}\pr{v_0}X_t^k }.
\end{equation*}
 By construction, $\sigma_M^2=\sum_{t\in\Z}
  \cov{\E{\sum_{k=1}^m\frac{\partial g}{\partial x_k}\pr{v_0}X_0^k\Big| \Fca_{-M}^M  }}{\E{\sum_{k=1}^m\frac{\partial g}{\partial x_k}\pr{v_0}X_t^k\Big| \Fca_{t-M}^{t+M}}}$ and a small calculation shows
\begin{align*}
& \abs{\sigma_M^2-\sigma^2}\\
 =& \abs{\sum_{t\in\Z}\cov{ \sum_{k=1}^m\frac{\partial g}{\partial x_k}\pr{v_0}\big(X_0^k-\E{X_0^k\mid \Fca_{-M}^M  }\big)}{\sum_{k=1}^m\frac{\partial g}{\partial x_k}\pr{v_0}\big(X_t^k-\E{X_t^k\mid \Fca_{t-M}^{t+M}  }\big)} }.
 \end{align*}
Note that 
\begin{align*}
& \sum_{i\in\Z}\delta_i\Bigg(\Bigg(\sum_{k=1}^m\frac{\partial g}{\partial x_k}\pr{v_0}\bigg(X_t^k-\E{X_t^k\Big| \Fca_{t-M}^{t+M}}\bigg)\Bigg)_{t\in\Z}\Bigg)\\
 \leq &\sum_{i\in\Z}\sum_{k=1}^m\frac{\partial g}{\partial x_k}\pr{v_0}\delta_i\bigg(\bigg(X_t^k-\E{X_t^k\Big| \Fca_{t-M}^{t+M}}\bigg)_{t\in\Z}\bigg)
 \leq 2\sum_{i\in\Z}\sum_{k=1}^m\frac{\partial g}{\partial x_k}\pr{v_0}\delta_i\big(\big(X_t^k\big)_{t\in\Z}\big),
\end{align*} 
which is finite by assumption \eqref{eq:condition_on_projectors}. We can thus apply Lemma~\ref{lem:conver_variance_UI} from the appendix to the above sum of covariances and obtain
  \begin{align*}
 \abs{\sigma_M^2-\sigma^2}
 & =  \lim_{n\to \infty}\frac{1}{n}\cdot \E{\pr{  \sum_{t=1}^n\sum_{k=1}^m\frac{\partial g}{\partial x_k}\pr{v_0}\pr{X_t^k-\E{X_t^k\Big| \Fca_{t-M}^{t+M}  }}  } ^2}\\
 &\leq  \bigg(\sum_{i\in\Z}
 \delta_i\bigg(\bigg(\sum_{k=1}^m\frac{\partial g}{\partial x_k}\pr{v_0}\pr{X_t^k-\E{X_t^k\Big| \Fca_{t-M}^{t+M}  }}\bigg)_{t\in\Z}    \bigg)\bigg)^2, 
 \end{align*}
where the last inequality follows by Lemma~\ref{lem:moments_stationary_sequence}. We have already shown above that the last expression is finite. Moreover, for each fixed $i$, it holds by the martingale convergence theorem that
$$ \lim_{M\rightarrow \infty} \delta_i\bigg(\bigg(\sum_{k=1}^m\frac{\partial g}{\partial x_k}\pr{v_0}\pr{X_t^k-\E{X_t^k\Big| \Fca_{t-M}^{t+M}  }}\bigg)_{t\in\Z}    \bigg)=0$$
and by dominated convergence, it thus follows $ \abs{\sigma_M^2-\sigma^2}\rightarrow 0$. 
\end{proof}

\begin{proof}[Proof of Lemma~\ref{lem:approximation_by_m_dep}]
Due to Chebychev's inequality, it suffices to prove
\begin{equation*}
 \lim_{M\to \infty}\limsup_{n\to \infty}
 \norm{R_{n,M}}_2=0.
\end{equation*}
By the definition of $R_{n,M}$, the above equality holds if the sum
$\sum_{N\geq 1}a_N$ converges, where
\begin{multline*}
a_N:=\sup_{n\geq 1}
 \frac{1}{b_n^{3/2}}
\left\lVert \sum_{1\leq j\neq k\leq b_n}
 \Big( h\pr{W^{\pr{N+1}}_{n,j},W^{\pr{N+1}}_{n,k}}-\E{h\pr{W^{\pr{N+1}}_{n,j},W^{\pr{N+1}}_{n,k} }  } \right.\\
\left.- \Big(h\pr{W^{\pr{N}}_{n,j},W^{\pr{N}}_{n,k}}-\E{h\pr{W^{\pr{N}}_{n,j},W^{\pr{N}}_{n,k} }  }\Big)\Big)\right\rVert_2.
\end{multline*}
 Splitting the supremum up into those cases of $n$ for which $\ell_n\geq 2(N+1)$ and those for which $\ell_n< 2(N+1)$, we have to prove the convergence of  $\sum_{N\geq 1}a_{N,1}$
and $\sum_{N\geq 1}a_{N,2}$, where
\begin{multline*}
a_{N,1}:=\sup_{n\geq 1, 2\pr{N+1}\leq \ell_n}
 \frac{1}{b_n^{3/2}}
\left\lVert \sum_{1\leq j\neq k\leq b_n}
 \Big( h\pr{W^{\pr{N+1}}_{n,j},W^{\pr{N+1}}_{n,k}}-\E{h\pr{W^{\pr{N+1}}_{n,j},W^{\pr{N+1}}_{n,k} }  } \right.\\
\left.- \Big(h\pr{W^{\pr{N}}_{n,j},W^{\pr{N}}_{n,k}}-\E{h\pr{W^{\pr{N}}_{n,j},W^{\pr{N}}_{n,k} }  }\Big)\Big)\right\rVert_2,
\end{multline*}
\begin{multline*}
a_{N,2}:=\sup_{n\geq 1, 2\pr{N+1}> \ell_n}
 \frac{1}{b_n^{3/2}}
\left\lVert \sum_{1\leq j\neq k\leq b_n}
 \Big( h\pr{W^{\pr{N+1}}_{n,j},W^{\pr{N+1}}_{n,k}}-\E{h\pr{W^{\pr{N+1}}_{n,j},W^{\pr{N+1}}_{n,k} }  } \right.\\
\left.- \Big(h\pr{W^{\pr{N}}_{n,j},W^{\pr{N}}_{n,k}}-\E{h\pr{W^{\pr{N}}_{n,j},W^{\pr{N}}_{n,k} }  }\Big)\Big)\right\rVert_2.
\end{multline*}
We bound $a_{N,1}$ by an application of Lemma~\ref{lem:moment_inequality_U_stats} from the appendix to obtain
\begin{equation*}
 a_{N,1}\leq C\sup_{n\geq 1, 2\pr{N+1}\leq \ell_n}\norm{W_{n,1}^{\pr{N+1}}-W_{n,1}^{\pr{N}} }_2.
\end{equation*}
In order to bound $a_{N,2}$, we notice that by assumption,
 $b_n<\ell_n\leq 2(N+1)\leq CN$ for $n$ large enough. Consequently, by the Lipschitz-continuity of $h$,
\begin{equation*}
 a_{N,2}\leq 2\sup_{n\geq 1, 2\pr{N+1}> \ell_n}\sqrt{b_n}\norm{W_{n,1}^{\pr{N+1}}-W_{n,1}^{\pr{N}} }_2\leq 2\sup_{n\geq 1, 2\pr{N+1}> \ell_n}\sqrt{N}\norm{W_{n,1}^{\pr{N+1}}-W_{n,1}^{\pr{N}} }_2.
\end{equation*}
Using the Lipschitz-continuity of $g\cdot\eta$ and  Lemma~\ref{lem:moments_stationary_sequence},
we derive that
\begin{equation*}
 \norm{W_{n,1}^{\pr{N+1}}-W_{n,1}^{\pr{N}} }_2\leq C\sum_{k=1}^m\sum_{i\in\Z}
 \delta_i\pr{\pr{\E{X_t^k\mid\Fca_{t-N-1}^{t+N+1} }-\E{X_t^k\mid\Fca_{t-N }^{t+N} } }_{t\in\Z}}.
\end{equation*}
The above summands are $0$ for
$\abs{i}\geq N+2$, while for $\abs{i}< N+2$,  
we get by similar arguments as in the proof of Lemma~\ref{lem:norm_approx_partial_sums} that
\begin{equation*}
  \delta_i\pr{\pr{\E{X_t^k\mid\Fca_{t-N-1}^{t+N+1} }-\E{X_t^k\mid\Fca_{t-N }^{t+N} } }_{t\in\Z}}
  \leq \delta_{N+1}\pr{ \pr{X_t^k}_{t\in\Z} }
  +\delta_{-N-1}\pr{ \pr{X_t^k}_{t\in\Z} }.
\end{equation*}
We thus obtain the bounds $$a_{N,1}\leq C N\sum_{k=1}^m\big(\delta_{N+1}\big( \pr{X_t^k}_{t\in\Z} \big)+
\delta_{-N-1}\big(\pr{X_t^k}_{t\in\Z} \big)\big)$$ and $$a_{N,2}\leq C N^{3/2}\sum_{k=1}^m\big(\delta_{N+1}\big( \pr{X_t^k}_{t\in\Z} \big)+\delta_{-N-1}\big(\pr{X_t^k}_{t\in\Z} \big)\big).$$
The convergence of  $\sum_{N\geq 1}a_{N,1}$ and $\sum_{N\geq 1}a_{N,2}$ then follows from
assumption \eqref{eq:condition_on_projectors}.
\end{proof}

\begin{proof}[Proof of Proposition \ref{Prop: Replacing centring Bernoulli}]
The proof is divided into two parts. First, we show that
 \begin{equation*}
 \lim_{n\to \infty}\sqrt{b_n}\pr{ \frac 1{b_n\pr{b_n-1}}\sum_{1\leq j\neq k\leq
b_n}\pr{\E{h\pr{W^{\pr{\eta}}_{n,j},W^{\pr{\eta}}_{n,k}}
 } -\E{h\pr{Z_{n,j},Z_{n,k}}
 } }                     }=0,
 \end{equation*}
where
 \begin{equation*}
 Z_{n,j}=\frac{1}{\sqrt{\ell_n}}\sum_{k=1}^m\frac{\partial g}{\partial x_k}\pr{v_0}\sum_{t\in B_{n,j}}\pr{X_t^k-\E{X_1^k}},
\end{equation*}
and afterwards that 
 \begin{equation*}
 \lim_{n\to \infty}\sqrt{b_n}\pr{ \frac 1{b_n\pr{b_n-1}}\sum_{1\leq j\neq k\leq
b_n} \E{h\pr{Z_{n,j},Z_{n,k}}
  }  -\E{  h\pr{Z_n,Z'_n }}                  }=0.
 \end{equation*}
The assertion then follows from Theorem \ref{thm:Bernoulli}. 

Starting with the first part, due to the Lipschitz-continuity of $h$ and stationarity, it suffices to show that
\begin{equation*}
 \lim_{n\to\infty}\sqrt{b_n}\E{ \abs{W^{\pr{\eta}}_{n,1}-Z_{n,1}  }}=0.
\end{equation*}
By the assumptions imposed on $g$ and $\eta$, there exists a constant $C$ such that for each $z=(z_k)_{k=1}^m\in\R^m$, it holds
\begin{equation*}
 \abs{(g\cdot \eta)\pr{z}-\sum_{k=1}^m\frac{\partial g}{\partial x_k }\pr{v_0}
 \pr{z_k-\E{X_1^k}}}\leq C\norm{ z-v_0 }_2.
\end{equation*}
With the choice $z=\pr{\frac{1}{\ell_n}\sum_{t\in B_{n,1}}X_t^k }_{k=1}^m$, this yields
\begin{equation*}
 \sqrt{b_n}\E{ \abs{W^{\pr{\eta}}_{n,j}-Z_{n,1}  }}
 \leq C\sqrt{b_n\ell_n}\sum_{k=1}^{m}
 \E{\pr{\frac{1}{\ell_n}\sum_{t\in B_{n,1}} \pr{X_t^k-\E{X_1^k}}   }^2}.
\end{equation*}
Lemma~\ref{lem:moments_stationary_sequence} from the appendix now implies that
$\sqrt{b_n}\E{ \abs{W^{\pr{\eta}}_{n,j}-Z_{n,1}  }}\leq C\sqrt{b_n/\ell_n}$,
which converges to $0$ by assumption.

Turning towards the second part, we define
\begin{equation*}
 A_{n,j}:=
  \frac{1}{\sqrt{\ell_n}}\sum_{k=1}^m\frac{\partial g}{\partial x_k}\pr{v_0}\sum_{t\in B_{n,j}}
 \pr{ \E{X_t^k\mid \sigma\pr{\varepsilon_u,\abs{u-t}\leq \ell_n^{\kappa}   }}-\E{X_1^k} }
\end{equation*}
and
\begin{equation*}
 A'_{n,j}:=\frac{1}{\sqrt{\ell_n}}\sum_{k=1}^m\frac{\partial g}{\partial x_k}\pr{v_0}\sum_{t\in B'_{n,j}}\pr{\E{X_t^k\mid \sigma\pr{\varepsilon_u,\abs{u-t}\leq \ell_n^{\kappa}   }}-\E{X_1^k}},
\end{equation*}
where
\begin{equation*}
 B'_{n,j}:=\ens{k\in \N, (j-1)\ell_n+1+\ell_n^{\kappa }\leq k\leq j\ell_n -\ell_n^{ \kappa}}.
\end{equation*}
Note that by construction, the sequence of random variables $\pr{A'_{n,j}}_{j\geq 1}$ is independent. Moreover, the following decomposition holds
\begin{align*}
& \sqrt{b_n}\abs{ \frac 1{b_n\pr{b_n-1}}\sum_{1\leq j\neq k\leq b_n} \E{h\pr{Z_{n,j},Z_{n,k}}} -\E{  h\pr{Z_n,Z'_n }}  }\\
  \leq & \sqrt{b_n}\frac 1{b_n\pr{b_n-1}}\sum_{1\leq j\neq k\leq b_n}\E{\abs{h\pr{Z_{n,j},Z_{n,k}}-h\pr{A'_{n,j},A'_{n,k}}}}\\
+ &\sqrt{b_n}\frac 1{b_n\pr{b_n-1}}\abs{\sum_{1\leq j\neq k\leq b_n}\pr{\E{h\pr{A'_{n,j},A'_{n,k}}}-\E{  h\pr{Z_n,Z'_n }} }}.
\end{align*}
By the Lipschitz-continuity of $h$ combined with stationarity, we can bound the first of the above terms by
$$ C\sqrt{b_n} \E{\abs{Z_{n,1}-A'_{n,1}  }}\leq C\sqrt{b_n} \E{\abs{Z_{n,1}-A_{n,1}  }}+C\sqrt{b_n} \E{\abs{A_{n,1}-A'_{n,1}  }}.$$ 
An application of Lemma \ref{lem:norm_approx_partial_sums} from the appendix yields 
\begin{align*}
\sqrt{b_n} \E{\abs{Z_{n,1}-A_{n,1}  }}\leq & C\sqrt{b_n}\sum_{k=1}^m \ell_n^{\kappa} \sum_{i: \abs{i}>\ell_n^{\kappa}}\delta_{i}\pr{\pr{ X_{t}^k}_{t\in\Z}}\\
\leq & C\sqrt{b_n/\ell_n}\ell_n^{\kappa/2}\sum_{k=1}^m\sum_{i: \abs{i}>\ell_n^{\kappa}}\ell_n^{\kappa/2+1/2}\delta_{i}\pr{\pr{ X_{t}^k}_{t\in\Z}}\\
\leq &C\sqrt{b_n/\ell_n}\ell_n^{\kappa/2}\sum_{k=1}^m\sum_{i: \abs{i}>\ell_n^{\kappa}}\abs{i}^{1/2+1/(2\kappa)} \delta_{i}\pr{\pr{ X_{t}^k}_{t\in\Z}},
\end{align*}
which tends  to  zero by the assumptions on the dependence coefficients and since $b_n=o(\ell_n^{1-\kappa})$.
By Lemma~\ref{lem:moments_stationary_sequence}, 
\begin{align*}
\sqrt{b_n}\E{\abs{A_{n,1}-A'_{n,1}}}\leq &C \frac{\sqrt{b_n}}{\sqrt{\ell_n}}\sum_{k=1}^m\norm{\sum_{t=1}^{\ell_n^{\kappa}} \E{X_t^k\mid \sigma\pr{\varepsilon_u,\abs{u-t}\leq \ell_n^{\kappa}   }}}_2\\
\leq & C \frac{\sqrt{b_n}}{\sqrt{\ell_n}}\sum_{k=1}^m \ell_n^{\kappa/2} \sum_{i\in\Z} \delta_i((\E{X_t^k\mid \sigma\pr{\varepsilon_u,\abs{u-t}\leq \ell_n^{\kappa}   }})_{t\in\Z})\\
\leq & C\frac{\sqrt{b_n}}{\sqrt{\ell_n}}\sum_{k=1}^m \ell_n^{\kappa/2} \sum_{i: \abs{i}\leq \ell_n^{\kappa}} \delta_i((X_t^k)_{t\in\Z}),
\end{align*}
which likewise tends  to  zero.

For the second of the above terms, we once more use the Lipschitz-continuity of $h$ and stationarity in order to obtain the upper bound $\sqrt{b_n}\E{\abs{A'_{n,1}-Z_n}}$, whose convergence has already been shown above since $Z_n=Z_{n,1}$.
\end{proof}

\begin{proof}[ Proof of Corollary~\ref{cor: replacing centring term bernoulli}]
It holds
\begin{align}
\nonumber & \sqrt{b_n}\pr{\E{\abs{Z_n-Z'_n}}-\sigma\E{\abs{N-N'}}  }
=  \sqrt{2}\sigma \sqrt{b_n}\pr{\frac{\E{\abs{Z_n-Z'_n}}}{\sqrt{2}\sigma}-\E{\abs{N}}}\\
=&  \sqrt{b_n}\E{\abs{Z_n-Z'_n}}\frac{\sqrt{\Var{(Z_n)}}-\sigma}{\sqrt{\Var{(Z_n)}}}
 + \sqrt{2}\sigma \sqrt{b_n}\pr{\frac{\E{\abs{Z_n-Z'_n}}}{\sqrt{2\Var{(Z_n)}}}-\E{\abs{N}}}.
 \label{eq:decomposition_centreing}
\end{align}
For the first of these terms, it suffices to  show that 
$\sqrt{b_n}\abs{\sqrt{\Var{(Z_n)}}-\sigma}\to 0$
since $\E{\abs{Z_n-Z'_n}}$ can be bounded by a constant independent of $n$ due to Lemma \ref{lem:moments_stationary_sequence} and since for $\Var{(Z_n)}\rightarrow \sigma^2>0$, it holds $\Var{(Z_n)}>\sigma^2/2$ for all $n$ large enough. 
Proceeding as in the proof of Proposition~2 in \citet*{MR2988107}, we derive that  for a centered time series $\pr{Y_t}_{t\in\Z}$ such that 
$Y_t=f\pr{\pr{\varepsilon_{t-u}}_{u\geq 0}}$, one has 
\begin{multline}
\abs{\operatorname{Var}\pr{\frac{1}{\sqrt{N}}\sum_{t=1}^NY_t}-\sum_{t\in\Z}\operatorname{Cov}\pr{Y_0,Y_t}}\\
\leq \sum_{j:\abs{j}>N}\sum_{i\in\Z}\delta_{i}\pr{\pr{Y_t}_{t\in\Z}}
\delta_{i-j}\pr{\pr{Y_t}_{t\in\Z}}
+2\sum_{j=1}^N \frac jN \sum_{i\in\Z}\delta_{i}\pr{\pr{Y_t}_{t\in\Z}}
\delta_{i-j}\pr{\pr{Y_t}_{t\in\Z}}.
\end{multline}
This follows from the following arguments. First, we expand $\operatorname{Var}\pr{\sum_{t=1}^NY_t/\sqrt{N}}$ in terms of $\operatorname{Cov}\pr{Y_0,Y_t}$ via a use of stationarity. Then we write 
\begin{equation}
\operatorname{Cov}\pr{Y_0,Y_t}=\sum_{i\in\Z}\E{P_i\pr{Y_0}P_i\pr{Y_t}}, 
\end{equation}
where $P_i\pr{Y}:=\E{Y\mid\sigma\pr{\varepsilon_u,u\leq i}}-
\E{Y\mid\sigma\pr{\varepsilon_u,u\leq i-1}}$. Finally, we use
the fact that $\abs{\E{P_i\pr{Y_0}P_i\pr{Y_t}}}\leq \norm{P_i\pr{Y_0}}_2
\norm{P_i\pr{Y_t}}_2$ due to the Cauchy-Schwarz inequality and 
Theorem~1 in \citet{MR2172215}.
Letting $N=\ell_n$, $Y_t=\sum_{k=1}^m\frac{\partial g}{\partial x_k}
\pr{v_0} \pr{X_t^k-\E{X_1^k}}$ and $a_i:=\max_{1\leq k\leq m}\delta_{i}\pr{\pr{X_t^k}_{t\in\Z}  }$, we infer that 
for some constant $C$ independent of $n$, 
\begin{multline*}
\sqrt{b_n}\abs{\Var{(Z_n)}-\sigma^2}
\leq C\sqrt{b_n} \sum_{k=1}^m\sum_{i\geq 0}\sum_{j=1}^{\ell_n}\frac{j}{\ell_n}\delta_{i}\pr{\pr{X_t^k}_{t\in\Z}  }\delta_{i-j}\pr{\pr{X_t^k}_{t\in\Z}}\mathbf{1}_{i\geq j}\\
+C\sqrt{b_n} \sum_{k=1}^m\sum_{i\geq 0}\sum_{j=\ell_n+1}^{\infty}\frac{j}{\ell_n}\delta_{i}\pr{\pr{X_t^k}_{t\in\Z}  }\delta_{i-j}\pr{\pr{X_t^k}_{t\in\Z}}\mathbf{1}_{i\geq j}.
\end{multline*}
Denote $A:=\sum_{i\geq 0}a_i$. Then elementary computations give 
\begin{align*}
\sum_{i\geq 0}\sum_{j=1}^{\ell_n}\frac{j}{\ell_n}a_ia_{i-j}\mathbf{1}_{i\geq j}&=
\sum_{j=1}^{\ell_n}\frac{j}{\ell_n}
\sum_{i\geq j}a_ia_{i-j} =\sum_{j=1}^{\ell_n}\frac{j}{\ell_n}
\sum_{k\geq 0}a_{k+j}a_k\\
&\leq \sum_{j=1}^{\ell_n}\frac{j}{\ell_n}
\sum_{k\geq 0}a_{k+j}A\leq A\sum_{j=1}^{\ell_n}\frac{j}{\ell_n}j^{-5/2}\sum_{k\geq 0}\pr{k+j}^{5/2}a_{j+k}
\leq C/\sqrt{\ell_n} 
\end{align*}
and 
\begin{align*}
\sum_{i\geq 0}\sum_{j=\ell_n+1}^{\infty}\frac{j}{\ell_n}a_ia_{i-j}\mathbf{1}_{i\geq j}&=
\sum_{i\geq \ell_n+1}\sum_{j=\ell_n+1}^{i}\frac{j}{\ell_n}a_ia_{i-j}  
\leq \sum_{i\geq \ell_n+1}\frac{i}{\ell_n}\sum_{j=\ell_n+1}^{i} a_ia_{i-j} \\&=\sum_{i\geq \ell_n+1}a_i\frac{i}{\ell_n}\sum_{k=0}^{i-\ell_n-1} a_{k} 
 \leq A\sum_{i\geq \ell_n+1}a_i\frac{i}{\ell_n}\leq C\ell_n^{-5/2}.
\end{align*}
Consequently, 
\begin{equation*}
\sqrt{b_n}\abs{\Var{(Z_n)}-\sigma^2}
\leq C\sqrt{b_n/\ell_n} , 
\end{equation*}
and the convergence of $\sqrt{b_n}\abs{\sqrt{\Var\pr{Z_n}}-\sigma  }\to 0$ follows from 
$$
\sqrt{b_n}\abs{\sqrt{\Var\pr{Z_n}}-\sigma  }=\sqrt{b_n}\abs{\Var{(Z_n)}-\sigma^2}\frac{1}{\sqrt{\Var\pr{Z_n}}+\sigma}.
$$
 
For the second term of \eqref{eq:decomposition_centreing}, we can  apply Corollary~2.6 in \citet{MR3502600} to obtain the bound $$ \sqrt{b_n}\pr{\frac{\E{\abs{Z_n-Z'_n}}}{\sqrt{2\Var{(Z_n)}}}-\E{\abs{N}}}\leq C\sqrt{b_n}\ell_n^{1-p/2} \pr{\log\ell_n}^{p/2},$$
 provided the assumptions stated in Corollary~\ref{cor: replacing centring term bernoulli} hold.  Indeed, 
$Z_n-Z'_n$ can be expressed as
\begin{equation*}
 Z_n-Z'_n=\frac{1}{\sqrt{\ell_n}}  \sum_{t=1}^{\ell_n} 
\phi\pr{\pr{\varepsilon_{t-u},\varepsilon'_{t-u} }_{u\in\N} },
\end{equation*}
where $\pr{\varepsilon'_u}_{u\in\Z}$ is an independent copy of $\pr{\varepsilon_u}_{u\in\Z}$ 
and 
\begin{equation*}
 \phi\pr{\pr{x_u,y_u }_{u\in\N} }:=\sum_{k=1}^m\frac{\partial g}{\partial x_k} 
\pr{v_0}\pr{f^k\pr{ \pr{x_u}_{u\in\N}}-f^k\pr{ \pr{y_u}_{u\in\N}}  }.
\end{equation*}
\end{proof}

\subsection{Verification of Examples}
\label{subsec:Verification examples}

\begin{proof}[Details of Example \ref{Example: Hoelder of linear process}]
Since $\varphi$ is $\gamma$-H\"older continuous, there exists a constant $C$ such that 
for each $x,y\in\R$,  $\abs{\varphi\pr{x}-\varphi\pr{y}}\leq C\abs{x-y}^\gamma$. In 
particular, $\abs{\varphi\pr{x}}\leq C\abs{x}^\gamma+\abs{\varphi\pr{0}}$ and it follows 
that
\begin{align*}
 \abs{\varphi^k\pr{x}-\varphi^k\pr{y}}&=
 \abs{\varphi\pr{x}-\varphi\pr{y}}
 \abs{\sum_{j=0}^{k-1}\varphi\pr{x}^j\varphi\pr{y}^{k-j-1}  }\\
 &\leq \max\ens{C,\abs{\varphi\pr{0}}}^{k-1}\abs{\varphi\pr{x}-\varphi\pr{y}}
  \sum_{j=0}^{k-1}\pr{1+\abs{x}^\gamma}^j
  \pr{1+\abs{y}^\gamma}^{k-j-1}\\
 &\leq C'\abs{x-y}^\gamma
\pr{1+\abs{x}^\gamma+\abs{y}^\gamma}^{k-1}.
\end{align*}
  For a fixed $i$, we thus have
\begin{align*}
& \delta_i\pr{\pr{X_t^k}_{t\in\Z} } 
 =\norm{\varphi\Big(\sum_{ j \in\Z  }a_{j}\varepsilon_{-j}\Big)^k-\varphi\Big(\sum_{ j \in\Z  }a_{j}\varepsilon_{-j}^{*,i}\Big)^k}_2\\
 \leq & C'\norm{\abs{a_{-i}}^\gamma\abs{\varepsilon_i-\varepsilon'_i}^\gamma
 \pr{1+\abs{\sum_{j\in\Z,j\neq-i}a_j\varepsilon_{-j}+a_{-i}\varepsilon_i}^\gamma+\abs{\sum_{j\in\Z,j\neq 
-i}a_j\varepsilon_{-j}+a_{-i}\varepsilon'_i}^\gamma  }^{k-1} }_2 \\
\leq & C'\norm{\abs{a_{-i}}^\gamma\abs{\varepsilon_i-\varepsilon'_i}^\gamma
 \pr{1+\abs{ a_{-i}\varepsilon_i}^\gamma+\abs{ a_{-i}\varepsilon'_i}^\gamma  }^{k-1}
 }_2\\
 +&C'\norm{\abs{a_{-i}}^\gamma\abs{\varepsilon_i-\varepsilon'_i}^\gamma
 \pr{1+\abs{\sum_{j\in\Z,j\neq -i}a_j\varepsilon_{-j} }^\gamma+\abs{\sum_{j\in\Z,j\neq 
-i}a_j\varepsilon_{-j} }^\gamma  }^{k-1}}_2\\
 \leq & C'\Big(\norm{\abs{\varepsilon_0}^{k\gamma} }_2+\norm{\abs{\varepsilon_0}^{(k-1)\gamma} }_2 \cdot \norm{\abs{\varepsilon_0}^{\gamma} }_2\Big)\abs{a_{-i}}^{k\gamma} +C'\abs{a_{-i}}^\gamma\norm{\abs{\varepsilon_0}^{\gamma} }_2\\
 +& C'\norm{\abs{\varepsilon_0}^{\gamma} }_2\abs{a_{-i}}^\gamma  \norm{  \abs{\sum_{j\in\Z,j\neq -i}a_j\varepsilon_{-j}}^{\pr{k-1}\gamma}   }_2,
\end{align*}
where the second inequality follows by a combination of the triangle inequality for $\gamma\in(0,1]$ and the $c_r$-inequality, and the third one is due to the independence of $\sum_{j\in\Z,j\neq -i}a_j\varepsilon_{-j}$ and $\pr{\varepsilon_{i},\varepsilon_{i}'}$.

Recall that we have to check that $\sum_{i\in\Z} i^2 \delta_i\pr{\pr{X_t^k}_{t\in\Z} } <\infty$ for each $k=1, \ldots, m$. By assumption, $\E{\abs{\varepsilon_0}^{2m\gamma}}<\infty$ and $\sum_{i\in\Z} i^2 \abs{a_i}^\gamma<\infty$, such that it remains to show $\E{\abs{ \sum_{j\in\Z,j\neq -i}a_j\varepsilon_{-j}}^{2\gamma\pr{k-1}}}<\infty$. In case $2\gamma\pr{k-1}\leq 1$, we can simply 
employ $$\E{\abs{\sum_{j\in\Z,j\neq 
-i}a_j\varepsilon_{-j}}^{2\pr{k-1}\gamma}}\leq \sum_{j\in\Z,j\neq 
0}\abs{a_j}^{2\pr{k-1}\gamma}\E{\abs{\varepsilon_0}^{2\pr{k-1}\gamma}},$$ which is finite by assumption. If $
2\gamma\pr{k-1}\in (1,2]$, then the Von Bahr-Esseen inequality gives (up to a constant) the same upper bound. 
If $2\gamma\pr{k-1}>2$, Rosenthal's inequality yields
\begin{multline*}
 \E{\abs{\sum_{j\in\Z,j\neq -i}a_j\varepsilon_{i-j}}^{2\pr{k-1}\gamma}}
\leq C \sum_{j\in\Z,j\neq 0}\abs{a_j}^{2\pr{k-1}\gamma}\E{\abs{\varepsilon_{0}}^{2\pr{k-1}\gamma}}
+C \pr{\sum_{j\in\Z,j\neq 0}a_j^{2}\E{ \varepsilon_{0}^{2 }}  }^{\pr{k-1}\gamma}.
\end{multline*}
\end{proof}

\begin{proof}[Details of Example \ref{Example: Functions of Gaussian proc}]
 This is a consequence of the estimation of the physical dependence measure 
 in Example~3 on pages 5967-5968 of \citet{MR3256190} applied separately to each function 
$\varphi^k$ for $1\leq k\leq m$.
\end{proof}

\begin{proof}[Details of Example \ref{Example: Volterra}]
 In order to give a bound on $\delta_i((X_t^k)_{t\in\Z})$ for a fixed $i$ 
and a $k\in\ens{1,\dots,m}$, we decompose $X_0$ as follows: Set $X_0=\varepsilon_i Y_i+Z_i$, 
where 
\begin{equation*}
 Y_i:=\sum_{j'\in\Z,j'\neq -i}a_{-i,j'} \varepsilon_{-j'}
 +\sum_{j\in\Z,j\neq -i}a_{j,-i} \varepsilon_{-j}
\end{equation*}
\begin{equation*}
 Z_i:=\sum_{j,j'\in\Z,j\neq j',j\neq -i, j'\neq -i}a_{j,j'}\varepsilon_{-j}\varepsilon_{-j'}.
\end{equation*}
 Thus, 
 \begin{equation*}
 X_0^k-\pr{X_0^{*,i}}^k=\sum_{\ell=0}^k\binom k\ell \pr{\varepsilon_i^\ell 
 Y_i^\ell Z_i^{k-\ell} -\pr{\varepsilon'_i}^\ell 
 Y_i^\ell Z_i^{k-\ell} }
\end{equation*}
and since the term with index $0$ vanishes, we derive that 
\begin{align*}
 \delta_i((X_t^k)_{t\in\Z}) &=\norm{ X_0^k-\pr{X_0^{*,i}}^k}_2\leq 
 \sum_{\ell=1}^k\binom k\ell \norm{\pr{\varepsilon_i^\ell-\pr{\varepsilon'_i}^\ell } Y_i^\ell 
Z_i^{k-\ell} }_2\\
&\leq 2 \sum_{\ell=1}^k\binom k\ell \norm{\varepsilon_0^\ell  }_2 \norm{Y_i^\ell 
Z_i^{k-\ell} }_2
\leq 2
 \sum_{\ell=1}^k\binom k\ell \norm{\varepsilon_0  }_{2\ell}^\ell \norm{Y_i}_{2k}^\ell 
\norm{Z_i  }_{2k}^{k-\ell},
\end{align*}
where the second inequality is due to the  independence of $\pr{\varepsilon_i,\varepsilon'_i}$ and 
$\pr{Y_i,Z_i}$, and the third inequality follows from an application of H\"older's inequality 
with conjugate exponents $k/\ell$ and $k/\pr{k-\ell}$ for $\ell\leq k-1$.
Following the arguments given on pages 2376-2377 in \citet*{MR4166203}, we obtain
\begin{equation*}
 \norm{Z_i  }_{2k}  \leq C
 \sqrt{\sum_{j,j'\in\Z\setminus\ens{-i},j\neq j' }a_{j,j'}^2}\norm{\varepsilon_0}_{2k}
 \leq C
 \sqrt{\sum_{j,j'\in\Z, j\neq j' }a_{j,j'}^2}\norm{\varepsilon_0}_{2k},
\end{equation*}
such that 
$ \norm{Z_i  }_{2k} $ can be bounded independently of $i$.
Moreover, an application of Rosenthal's inequality yields
\begin{equation*}
 \norm{Y_i}_{2k}\leq C\sqrt{\sum_{j,j'\in\Z, j\neq j' }a_{j,j'}^2}\cdot\sqrt{\sum_{j\in\Z, j\neq -i}
 \pr{a_{-i,j}^2+a_{j,-i}^2 }}\norm{\varepsilon_0}_{2k}.
\end{equation*}
Thus, $\delta_{i}\pr{\pr{X_t^k}_{t\in\Z}}\leq C\sqrt{\sum_{j\in\Z, j \neq -i}
 \pr{a_{-i,j}^2+a_{j,-i}^2 }}$ and the result
follows.
\end{proof}

\begin{appendices}
\section{Appendix}

\subsection{Auxiliary results for functionals of i.i.d.\ sequences}
This appendix collects some auxiliary results for functionals of i.i.d.\ sequences, which we require for our proofs. We consider sequences $\pr{X_t}_{t\in\Z}$ of the form $X_t:=f\pr{\pr{\varepsilon_{t-u}}_{u\in\Z}   }$, where 
$f\colon\R^{\Z}\to\R$ is measurable, $\pr{\varepsilon_u}_{u\in \Z}$ is an i.i.d.\ sequence and $\E{X_t}=0$. Denote $\Fca_{M}^N:=\sigma\pr{\varepsilon_u,M\leq u\leq N}$.
To quantify the dependence, let 
$\pr{\varepsilon'_u}_{u\in\Z}$ denote an independent copy of  $\pr{\varepsilon_u}_{u\in\Z}$ and 
define 
\begin{equation*}
 \delta_{i}\pr{\pr{X_t}_{t\in\Z}}:=\norm{X_0-X_0^{*,i}}_2,
\end{equation*}
where $X_0^{*,i}=f\pr{ \pr{\varepsilon^{*,i}_{-u  }}_{u\in\Z} }$ and $\varepsilon^{*,i}_v=\varepsilon'_i$ if 
$v=i$ 
and $\varepsilon^{*,i}_v=\varepsilon_v$ otherwise.

We start by presenting a bound on the partial sum of $\pr{X_t}_{t\in\Z}$, which is a special case of Proposition~1 
in \citet{MR2988107}.
\begin{lemma}\label{lem:moments_stationary_sequence}
The following inequality holds for all $N\in \N$:
\begin{equation*}
 \norm{\sum_{t=1}^N X_t }_2\leq \sqrt{N}
 \sum_{i\in\Z}\delta_i\pr{\pr{X_t}_{t\in\Z}  }.
\end{equation*}

\end{lemma}

\begin{lemma}\label{lem:conver_variance_UI}
 Suppose that $\sum_{i\in\Z}\delta_i\big(\pr{X_t}_{t\in\Z}\big)<\infty$. Then the sequence 
 $
  \big( \frac 1N\big(\sum_{t=1}^NX_t   \big)^2\big)_{N\geq 1}
 $
is uniformly integrable. Moreover, the
series $\sum_{t\in\Z}\abs{\cov{X_0}{X_t} }$ converges and 
\begin{equation*}
\lim_{N\to \infty}\frac 1N\E{\pr{\sum_{t=1}^NX_t }^2 }
= \sum_{t\in\Z} \operatorname{Cov} 
\pr{X_0,X_t} .
\end{equation*}
\end{lemma}
\begin{proof}
 The convergence of the series is established in Proposition~2 of 
\citet{MR2988107}. It remains to check the uniform integrability of  $
  \big( \frac 1N\big(\sum_{t=1}^NX_t   \big)^2\big)_{N\geq 1} $. Since the sequence is bounded in $\mathbb L^1$, we only have to check
 that 
 \begin{equation*}
  \lim_{\delta\to 0}\sup_{A:\PP\pr{A}<\delta}\E{\frac 1N\pr{\sum_{t=1}^NX_t   
}^2\mathbf{1}_A }=0.
 \end{equation*}
 To do so, let $X_t^{\pr{M}}:=\E{X_t\mid \Fca_{t-M}^{t+M}}$. Then 
 \begin{align*}
 & \sup_{A:\PP\pr{A}<\delta}\E{\frac 1N\pr{\sum_{t=1}^NX_t}^2\mathbf{1}_A }\\
\leq &2\sup_{A:\PP\pr{A}<\delta}\E{\frac 
1N\pr{\sum_{t=1}^NX_t^{\pr{M}} 
}^2\mathbf{1}_A }+2 \E{\frac 1N\pr{\sum_{t=1}^N\pr{X_t-X_t^{\pr{M}}}   
}^2  }
 \end{align*}
and using Lemma~\ref{lem:moments_stationary_sequence}, we obtain
\begin{align*}
   \sup_{A:\PP\pr{A}<\delta}\E{\frac 1N\pr{\sum_{t=1}^NX_t  
}^2\mathbf{1}_A }\leq 
&  2\sup_{A:\PP\pr{A}<\delta}\E{\frac 
1N\pr{\sum_{t=1}^NX_t^{\pr{M}} 
}^2\mathbf{1}_A }\\
+ &2\pr{\sum_{i\in \Z}\delta_i\pr{\pr{X_t-\E{X_t\mid \Fca_{t-M}^{t+M}}}_{t\in\Z}} }^2.
\end{align*}
Employing the uniform integrability of $
  \big( \frac 1N\big(\sum_{t=1}^NX_t^{(M)}   \big)^2\big)_{N\geq 1}
 $ for any fixed $M$, we derive that
\begin{equation*}
 \limsup_{\delta\to 0}\sup_{A:\PP\pr{A}<\delta}\E{\frac 1N\pr{\sum_{t=1}^NX_t
}^2\mathbf{1}_A }\leq 2\pr{\sum_{i\in \Z}\delta_i\pr{\pr{X_t-\E{X_t\mid \Fca_{t-M}^{t+M}}}_{t\in\Z}} }^2.
\end{equation*}
Since $\delta_i\pr{\pr{X_t-\E{X_t\mid \Fca_{t-M}^{t+M}}}_{t\in\Z}}\leq 2\delta_i\pr{\pr{X_t }_{t\in\Z}}$, we conclude by an application of the dominated convergence theorem.
\end{proof}

There moreover holds a central limit theorem for $\pr{X_t}_{t\in\Z}$ (see Theorem~1 in \citet{MR2988107}):

\begin{lemma}\label{lem:TLC_suite_bernoullienne}
 Suppose that $\sum_{i\in\Z}\delta_i\pr{\pr{X_t}_{t\in\Z}}<\infty$. Then the following convergence in distribution holds
\begin{equation*}
 \frac 1{\sqrt{n}}\sum_{t=1}^nX_t\rightarrow N\pr{0,\sigma^2}, 
\end{equation*}
where 
\begin{equation*}
 \sigma^2=\sum_{t\in\Z} \cov{X_0}{X_t}.
\end{equation*}
\end{lemma}

We also require an estimate on the $\mathbb L^2$-norm of
partial sums of $\big(X_t-\E{X_t\mid \mathcal{F}_{t-M}^{t+M}}\big)_{t\geq 1}$ for some $M\in\N$. 
\begin{lemma}\label{lem:norm_approx_partial_sums}
It holds
 \begin{equation*}
  \norm{\sum_{t=1}^N\pr{X_t-\E{X_t\mid \mathcal{F}_{t-M}^{t+M}  }}}_2\leq (4M+3)
  \sqrt{N}\sum_{i:\abs{i}\geq M}\delta_i\pr{ \pr{X_t}_{t\in\Z}  }.
 \end{equation*}
\end{lemma}
\begin{proof}
Lemma~\ref{lem:moments_stationary_sequence} applied to
$X_t-\E{X_t\mid \mathcal{F}_{t-M}^{t+M}  }$ yields
 \begin{equation*}
  \norm{\sum_{t=1}^N\pr{X_t-\E{X_t\mid \mathcal{F}_{t-M}^{t+M}  }}  }_2\leq
  \sqrt{N}
  \sum_{i\in\Z}\delta_i\pr{\pr{X_t-\E{X_t\mid \mathcal{F}_{t-M}^{t+M}  }}_{t\in\Z}  }.
 \end{equation*}
For $\abs{i}\geq M+1$, we use $\delta_i\pr{\pr{X_t-\E{X_t\mid \mathcal{F}_{t-M}^{t+M}  }}_{t\in\Z}  }\leq \delta_i\pr{\pr{X_t }_{t\in\Z}  }$, whereas for $\abs{i}\leq M$, we
employ $\delta_i\pr{\pr{X_t-\E{X_t\mid \mathcal{F}_{t-M}^{t+M}  }}_{t\in\Z}  } \leq 2\norm{X_0-\E{X_0\mid \mathcal{F}_{-M}^{M}  }}_2, $  thus obtaining 
  \begin{align}\label{eq: Lemma Appendix Approx}
  &\norm{\sum_{t=1}^N\pr{X_t-\E{X_t\mid \mathcal{F}_{t-M}^{t+M}  }}  }_2\\
\nonumber  \leq &   2\pr{2M+1}\sqrt{N}   \norm{X_0-\E{X_0\mid \mathcal{F}_{-M}^{M}  }}_2+  \sqrt{N}  \sum_{i\in\Z:\abs{i}\geq M+1}\delta_i\pr{\pr{X_t }_{t\in\Z}  }.
 \end{align}
In the following, we derive an upper bound for the first of the above terms.  By the martingale convergence theorem, it holds
 \begin{equation*}
   \norm{X_0-\E{X_0\mid \mathcal{F}_{-M}^{M}  }}_2^2
   \leq \sum_{i\geq M+1}
   \norm{\E{X_0\mid \mathcal{F}_{-i}^i}-\E{X_0\mid \mathcal{F}_{-(i-1)}^{i-1}}  }_2^2.
 \end{equation*}
Moreover,  
 \begin{align*}
 &  \norm{\E{X_0\mid \mathcal{F}_{-i}^i}-\E{X_0\mid \mathcal{F}_{-(i-1)}^{i-1}}  }_2 \\
   \leq &  \norm{\E{X_0\mid \mathcal{F}_{-i}^i}-\E{X_0^{*,i}\mid \mathcal{F}_{-i}^i}}_2+\norm{\E{X_0^{*,i}\mid \mathcal{F}_{-i}^i}-\E{X_0\mid \mathcal{F}_{-(i-1)}^{i-1}}  }_2\\
\leq &  \delta_i\pr{\pr{X_t}_{t\in\Z}}+  \delta_{-i}\pr{\pr{X_t}_{t\in\Z}},
 \end{align*}
 where the second inequality follows from 
  \begin{equation*}
  \E{X_0^{*,i}\mid \mathcal{F}_{-i}^i }
 =\E{X_0^{*,i}\mid  \mathcal{F}_{-i}^{i-1}}
 =\E{X_0 \mid \mathcal{F}_{-i}^{i-1}}
 \end{equation*}
 combined with 
   \begin{equation*}
  \E{X_0 \mid   \mathcal{F}_{-(i-1)}^{i-1}}
   =\E{X_0^{*,i}\mid   \mathcal{F}_{-(i-1)}^{i-1}}
   =\E{X_0^{*,-i}\mid   \mathcal{F}_{-i}^{i-1} }.
 \end{equation*}
By the comparison of the $\ell^1$- and $\ell^2$-norm, we thus have 
 \begin{equation*}
 \norm{X_0-\E{X_0\mid \mathcal{F}_{-M}^{M}  }}_2\leq  \sum_{i:\abs{i}\geq M+1} \delta_i\pr{\pr{X_t}_{t\in\Z}}.
 \end{equation*}
Inserting the above bound into \eqref{eq: Lemma Appendix Approx} concludes the proof.
\end{proof}

\subsection{A moment inequality for $U$-statistics}

\begin{lemma}\label{lem:moment_inequality_U_stats}
Let $\pr{\varepsilon_u}_{u\in\Z}$ be an i.i.d.\ sequence.
Let $M\geq 0$ and $\ell >2M$  be integers. Define the random vectors $V_j$ by
$V_j:=\pr{\varepsilon_u}_{u=j\ell+1-M}^{\pr{j+1}\ell+M}$. Let $h\colon
\R^2\to \R$ be a Lipschitz-continuous function,
let $f_1,f_2\colon \R^{\ell+2M}\to\R$ be measurable
functions and let $U_N$ be defined by
\begin{equation*}
 U_N:=\sum_{1\leq j<k\leq N}\pr{
 h\pr{f_1\pr{V_j},f_1\pr{V_k}}
 - h\pr{f_2\pr{V_j},f_2\pr{V_k}}}.
\end{equation*}
Then the following inequality holds
\begin{equation*}
 N^{-3/2}\norm{U_N-\E{U_N}}_2\leq C
 \norm{f_1\pr{V_0}-f_2\pr{V_0} }_2,
\end{equation*}
where $C$ is a constant depending only on $h$.
\end{lemma}
\begin{proof}
 The difficulty here lies in the fact that the  vectors $V_j, j\geq 1$, are not independent.
 Denote $$H_{j,k}:=h\pr{f_1\pr{V_j},f_1\pr{V_k}} - h\pr{f_2\pr{V_j},f_2\pr{V_k}}.$$
 We will prove the inequality
 \begin{equation*}
 \label{eq:moment_inequality_U_stats_pair}
  \norm{U_{N}-\E{U_{N}}}_2\leq CN^{3/2}\sup_{k\geq 1}\norm{H_{0,k}}_2,
 \end{equation*}
 from which the assertion then follows by the Lipschitz-continuity of $h$. To verify the above inequality, we will distinguish between the cases where $N$ is even and those where $N$ is odd.
 Let us first consider even values of $N$, in which case we can write $2N$ instead of $N$. Denote
 by $\Fca_k$ the $\sigma$-algebra generated by the
 random variables $V_{k'}$ for $k'\leq k$.
 Then it holds
 \begin{align*}
 & \norm{U_{2N}-\E{U_{2N}}}_2 
  =
  \norm{\sum_{k=2}^{2N}\sum_{j=1}^{k-1}
  \pr{H_{j,k}-\E{H_{j,k}\mid \Fca_{k-2}}
  } +\sum_{k=2}^{2N}\sum_{j=1}^{k-1}\pr{
   \E{H_{j,k}\mid \Fca_{k-2}}-
   \E{H_{j,k}} }}_2\\
\leq &   \norm{\sum_{i=1}^{N}\sum_{j=1}^{2i-1}
  \pr{H_{j,2i}-\E{H_{j,2i}\mid \Fca_{2i-2}}
  }}_2
  +\norm{\sum_{i=1}^{N-1}\sum_{j=1}^{2i}
  \pr{H_{j,2i+1}-\E{H_{j,2i+1}\mid \Fca_{2i+1-2}}
  }}_2\\
  &+\norm{\sum_{k=2}^{2N}\sum_{j=1}^{k-2}\pr{
   \E{H_{j,k}\mid \Fca_{k-2}}-
   \E{H_{j,k}}
  }}_2
  +\norm{\sum_{k=2}^{2N} \pr{
   \E{H_{k-1,k}\mid \Fca_{k-2}}-
   \E{H_{k-1,k}}
  }}_2.
 \end{align*}
  For the first two terms, we additionally define $d_i:=
\sum_{j=1}^{2i-1}
  \pr{H_{j,2i}-\E{H_{j,2i}\mid \Fca_{2i-2}}
  }$ and $d'_i:= \sum_{j=1}^{2i}
  \pr{H_{j,2i+1}-\E{H_{j,2i+1}\mid \Fca_{2i+1-2}}
  }$, such that the sequences $\pr{d_i,\Fca_{2i}}_{i\geq 1}$ and
  $\pr{d'_i,\Fca_{2i+1}}_{i\geq 1}$ are martingale
  differences. For the third term, we use the independence between $V_k$ and $\Fca_{k-2}$ to get  that
  $\E{H_{j,k}\mid \Fca_{k-2}}=
  \E{H_{j,-1}\mid V_j}$, and we have to bound
  the moments of a two-dependent identically distributed
centred sequence. For the fourth term, we simply use the 
triangle inequality. 
By orthogonality of $\pr{d_i}_{i\geq 1}$ and orthogonality of $\pr{d'_i}_{i\geq 1}$, it follows
 \begin{align*}\label{eq:bound_moment_order_two}
  \norm{U_{2N}-\E{U_{2N}}}_2 
  \leq
&  \pr{\sum_{i=1}^{N}\norm{\sum_{j=1}^{2i-1}
  \pr{H_{j,2i}-\E{H_{j,2i}\mid \Fca_{2i-2}}
  }}_2^2}^{1/2}\\
 & +\pr{\sum_{i=1}^{N-1}\norm{\sum_{j=1}^{2i}
  \pr{H_{j,2i+1}-\E{H_{j,2i+1}\mid \Fca_{2i+1-2}}
  }}_2^2}^{1/2}\\
 & +\norm{\sum_{k=2}^{2N}\sum_{j=1}^{k-2}\pr{
   \E{H_{j,-1}\mid V_j}-
   \E{H_{j,-1}
  }}}_2
  +4N\sup_{k\geq 1}\norm{H_{0,k}     }_2.
 \end{align*}
The first of the above terms can be further bounded via
\begin{align*}
&\pr{\sum_{i=1}^{N} \norm{\sum_{j=1}^{2i-1} \pr{H_{j,2i}-\E{H_{j,2i}\mid \Fca_{2i-2}} }}_2^2}^{1/2}
 \leq\pr{\sum_{i=1}^{N}\pr{  \sum_{j=1}^{2i-1}\norm{  \pr{H_{j,2i}-\E{H_{j,2i}\mid \Fca_{2i-2}}}}_2  }^2}^{1/2}\\
 &\leq\pr{\sum_{i=1}^{N} 4\pr{  \sum_{j=1}^{2i-1}\norm{ H_{j,2i}}_2  }^2}^{1/2} 
  \leq \pr{\sum_{i=1}^{N}16 i^2\sup_{k\geq 1}\norm{H_{0,k}     }_2^2}^{1/2}
 \leq 4 N^{3/2}\sup_{k\geq 1}\norm{H_{0,k}}_2.
\end{align*}
The second term can be treated analogously. In order to
bound the third term, we switch the sums over
$j$ and $k$ to obtain

  \begin{align*}
 & \norm{\sum_{k=2}^{2N}\sum_{j=1}^{k-2}\pr{ \E{H_{j,-1}\mid V_j}- \E{H_{j,-1}}}}_2\\
 = & \norm{\sum_{j=1}^{2N-2}\pr{2N-j+1}(\E{H_{j,-1}\mid V_j}-
   \E{H_{j,-1}})}_2\\
   \leq & \norm{\sum_{i=1}^{N-1}\pr{2N-2i+1}(\E{H_{2i,-1}\mid V_{2i}}-
   \E{H_{2i,-1}})}_2\\
   &+\norm{\sum_{i=1}^{N-1}\pr{2N-\pr{2i-1}+1}
   (\E{H_{2i-1,-1}\mid V_{2i-1}}-
   \E{H_{2i-1,-1}})}_2\\
   =& \sqrt{ \sum_{i=1}^{N-1}\pr{2N-2i+1}^2
   \norm{Y(\E{H_{0,-1}\mid V_0}-
   \E{H_{0,-1}})}_2^2     }\\
   &+   \sqrt{ \sum_{i=1}^{N-1}\pr{2N-\pr{2i-1}+1}^2
   \norm{(\E{H_{0,-1}\mid V_0}-
   \E{H_{0,-1}})}_2^2     }\\
    \leq & C N^{3/2}\sup_{k\geq 1}\norm{H_{0,k}}_2.
  \end{align*}
This proves
$  \norm{U_{2N}-\E{U_{2N}}}_2\leq CN^{3/2}\sup_{k\geq 1}\norm{H_{0,k}}_2$.
In order to show the corresponding inequality for the index $2N+1$ instead of $2N$, we note that $U_{2N+1}-\E{U_{2N+1}}$ differs from
$U_{2N}-\E{U_{2N}}$ only by the term
$ \sum_{j=1}^{2N}\pr{H_{j,2N+1}-
\E{H_{j,2N+1}}}$, whose $\mathbb L^2$-norm is
smaller than $4N\sup_{k\geq 1}\norm{H_{0,k}}_2$.
This ends the proof of Lemma~\ref{lem:moment_inequality_U_stats}.
\end{proof}

\subsection{Tools for the proof of Lemma~\ref{lem:conv_m_dep}}
\begin{lemma}\label{lem:convergence_of_sigma_n}
 Let $\pr{Y_n}_{n\geq 1}$ be a sequence of random variables such that $\pr{Y_n^2}_{n\geq 1}$ is uniformly integrable and 
 $Y_n \to  N\pr{0,\sigma^2}$ in distribution with $\sigma>0$. Let
$Y'_n$ and $Y''_n$  be independent copies of $Y_n$ 
 and let $h \colon\R^2\to \R$ be a Lipschitz-continuous function. Then 
 \begin{equation*}
  \lim_{n\to\infty}\cov{h\pr{Y_n,Y'_n}  }{h\pr{Y_n,Y''_n}}  = 
\cov{h\pr{N,N'}  }{h\pr{N,N''}} ,
 \end{equation*}
where $N$, $N'$, $N'' $ are independent $N(0,\sigma^2)$-distributed random 
variables.
\end{lemma}
\begin{proof}
  By independence,  the sequence of random vectors $\pr{Y_{n},Y'_{n},Y''_{n}}$ converges in 
distribution to $ \pr{N,N',N''}$. By Skorohod's representation theorem, there exist a probability space
 $(\tilde{\Omega},\tilde{\Fca},\tilde{\PP})$, sequences of random variables  $\pr{Z_{n}}_{n\geq 
1}$,  $\pr{Z'_{n}}_{n\geq 1}$ and $\pr{Z''_{n}}_{n\geq 1}$ and random variables $Z$, $Z'$ and $Z''$, each defined on $\tilde{\Omega}$, such that  for all $n\geq 1$, the vectors $\pr{Y_{n},Y'_{n},Y''_n}$ and 
$\pr{Z_{n},Z'_{n},Z''_n }$ have 
 the same distribution, $\pr{Z,Z',Z''}$ has the same distribution as $\pr{N,N',N''}$, and the sequence $\pr{Z_{n}}_{n\geq 1}$ (respectively  $\pr{Z'_{n}}_{n\geq 1}$ and $\pr{Z''_{n}}_{n\geq 1}$) 
 converges to $Z$ (respectively $Z'$ and $Z''$) almost surely.
Note that for each fixed $n$, it holds
\begin{equation*}
 \cov{h\pr{Y_n,Y'_n}  }{h\pr{Y_n,Y''_n}}=
 \cov{h\pr{Z_n,Z'_n}  }{h\pr{Z_n,Z''_n}}
\end{equation*}
as well as 
  \begin{equation*}
  \cov{h\pr{N,N'}  }{h\pr{N,N''}} =\cov{h\pr{Z,Z'}  }{h\pr{Z,Z''}}.
  \end{equation*}
Due to the elementary fact that $\cov{U_n}{V_n} \to \cov{U}{V}$ if $U_n\to U$ and 
$V_n\to V$ in $\mathbb L^2$, it hence suffices to show  
\begin{equation*}
 \norm{h\pr{Z_n,Z'_n} -h\pr{Z,Z'}  }_2\to 0.
\end{equation*}
Since $h$ is Lipschitz-continuous and the sequence
$\pr{Z_n^2+\pr{Z'_n}^2+Z^2+\pr{Z'}^2}_{n\geq 1}$ is uniformly integrable, the 
sequence 
$\pr{\pr{h\pr{Z_n,Z'_n} -h\pr{Z,Z'}}^2}_{n\geq 1}$ is uniformly integrable as well. 
By the continuity of $h$, 
the sequence $\pr{\pr{h\pr{Z_n,Z'_n} -h\pr{Z,Z'}}^2}_{n\geq 1}$ converges to
$0$ almost surely. Combined, this yields the desired $\mathbb L^2$-convergence and finishes the proof.
 \end{proof}
\end{appendices}

\textbf{Acknowledgment:} The authors would like to thank the referee for his/her comments that improved the presentation of the paper.

\def\cprime{$'$} \def\polhk#1{\setbox0=\hbox{#1}{\ooalign{\hidewidth
  \lower1.5ex\hbox{`}\hidewidth\crcr\unhbox0}}} \def\cprime{$'$}

\end{document}